\documentclass[11pt]{article}

\usepackage{geometry}
\usepackage{graphicx}
\usepackage{bm}
\usepackage{epstopdf}
\usepackage{booktabs} % for much better looking tables
\usepackage{array} % for better arrays (eg matrices) in maths
\usepackage{paralist} % very flexible & customisable lists (eg. enumerate/itemize, etc.)
\usepackage{verbatim} % adds environment for commenting out blocks of text & for better verbatim
\usepackage{subfigure} % deprecated, use caption and subcaption
\usepackage{caption}
\usepackage{fancyhdr} % This should be set AFTER setting up the page geometry
\usepackage[nottoc,notlof,notlot]{tocbibind}
\usepackage[titles,subfigure]{tocloft}

\usepackage{amsmath,amsfonts,amsthm,mathrsfs,amssymb,cite}
\usepackage[usenames]{color}
\usepackage{mathtools}
\mathtoolsset{showonlyrefs}

\newtheorem{thm}{Theorem}[section]

\newtheorem{lem}{Lemma}[section]

\theoremstyle{definition}

\newtheorem{rem}{Remark}[section]
\newtheorem{example}{Example}
\numberwithin{equation}{section}

\title{\bf  A Computationally Efficient Finite Element Method for Shape Reconstruction of Inverse Conductivity Problems}

\author{Lefu Cai \thanks{School of Mathematics, Harbin Institute of Technology, Harbin, P. R. China. Email: {\tt 25B312009@stu.hit.edu.cn}},
\and Zhixin Liu\thanks{School of Mathematics, Harbin Institute of Technology, Harbin, P. R. China. Email: {\tt zxliu24@163.com}},
\and Minghui Song \thanks{School of Mathematics, Harbin Institute of Technology, Harbin, P. R. China. Email: {\tt songmh@hit.edu.cn}},
\and Xianchao Wang\thanks{School of Mathematics, Harbin Institute of Technology, Harbin, P. R. China. Email: {\tt xcwang90@gmail.com}}}

\date{} % Activate to display a given date or no date (if empty),
\begin{document}
\maketitle

\begin{abstract}

The inverse conductivity problem aims at determining the unknown conductivity inside a bounded domain from boundary measurements. In practical applications, algorithms based on minimizing a regularized residual functional subject to PDE constraints have been widely used to deal with this problem.  However, such approaches typically require repeated iterations and solving the forward problem at each iteration, which leads to a heavy computational cost. To address this issue, we first reformulate the inverse conductivity problem as a minimization problem involving a regularized residual functional. We then transform this minimization problem into a variational problem and establish the equivalence between them. This reformulation enables the employment of the finite element method to reconstruct the shape of the object from finitely many measurements. Notably, the proposed approach allows us to identify the object directly without requiring any iterative procedure.  {\it A prior} error estimates are rigorously established to demonstrate the theoretical soundness of the finite element method. Based on these estimates, we provide a criterion for selecting the regularization parameter.  Additionally, several numerical examples are presented to verify the feasibility of the proposed approach in shape reconstruction.

\medskip

\medskip


%\noindent{\bf 2010 Mathematics Subject Classification:}~~35R30, 35P25, 78A46

\end{abstract}

\section{Introduction}
This paper is concerned with an inverse conductivity problem for the elliptic partial differential equation in a bounded domain, namely, to determine the shape of the conductivity coefficient from boundary measurements. This type of  problem has wide applications in electrical impedance tomography (EIT), including lungs ventilation\cite{GATIJMG},  breast tissue imaging\cite{IF, Ammari2007},  and brain imaging\cite{Badia}.  Next, we present the mathematical formulation of the inverse conductivity problem for our study.

Let $\Omega$  be a bounded open connected domain of $\mathbb{R}^d$, $d\geq 2$, with a smooth boundary $\partial \Omega$ and  an outer normal vector $\nu$.  We consider  the following elliptic equation with a Neumann boundary condition
\begin{equation}\label{PDE}
\left\{
\begin{aligned}
&\nabla\cdot(\sigma\nabla u)=0            &\mbox{in}&\enspace\Omega,\\
&\sigma\frac{\partial u}{\partial\nu}=g  &\mbox{on}&\enspace\partial\Omega,
\end{aligned}\right.
\end{equation}
where $u$ represents the electric potential,   $\sigma$ is the isotropic electric conductivity,  and $g$ signifies the electric current density.
The forward problem of \eqref{PDE} is to determine the electric potential $u$ for a given conductivity $\sigma$ and boundary input $g$.
%It is worth noting that the solution to \eqref{PDE} is not unique, as any two solutions differ by a constant. To guarantee uniqueness, we impose the constraint that the solution $u$ has zero mean on $\partial\Omega$, that is, $\int_{\partial\Omega}u\, \mathrm{d}s=0$.
 %Under this condition, the existence and uniqueness of the solution are guaranteed by the Lax-Milgram theorem.
 %To this end,  let  $L_\diamond^2(\partial\Omega)$  denote the subspaces of $L^2(\partial\Omega)$ with vanishing integral mean on $\partial \Omega$, and  let $L^\infty_+(\Omega)$ denote the subspace of $L^{\infty}(\Omega)$ with positive essential infima.
We  define the Neumann-to-Dirichlet (NtD) operator by %$\Lambda(\sigma):L_\diamond^2(\partial\Omega)\rightarrow L_\diamond^2(\partial\Omega)$ by
\begin{equation*}
\Lambda(\sigma):g\mapsto u|_{\partial\Omega}.
\end{equation*}
The classical {\it Calder\'on problem} consists in reconstructing the unknown conductivity $\sigma$ from the NtD operator $\Lambda(\sigma)$.
It is well known that  the {\it Calder\'on problem} is a highly nonlinear and severely ill-posed inverse problem, and its reconstruction requires infinite-dimensional boundary measurements.
To incorporate the case of finitely many measurements, we assume that the Galerkin projection of $\Lambda(\sigma)$ onto the dual space of span$\{g_1,\dots,g_m\}$ is available, that is, we can measure the symmetric matrix
\begin{equation*}
F(\sigma)=\left(\int_{\partial \Omega}\,  g_i \Lambda (\sigma ) g_j \,  \mathrm{d}s\right)_{i,j=1}^m \in \mathbb{S}_m \subset \mathbb{R}^{m\times m}.
\end{equation*}
Furthermore, we assume that the  unknown conductivity  admits a piecewise-constant on a given resolution, i.e.,   $\sigma = (\sigma_1,\dots,\sigma_M)\in \mathbb{R}_+^M$.
Hence, the inverse problem under consideration is to determine the conductivity vector $\sigma \in \mathbb{R}_+^M$ from the finitely many measurements $F(\sigma)\in \mathbb{R}^{m\times m}$, namely,
 \begin{equation*}
   F(\sigma)\in \mathbb{R}^{m\times m} \longmapsto \sigma\in \mathbb{R}_+^M.
 \end{equation*}

Theoretically, the uniqueness results for inverse conductivity problems have been extensively investigated in the infinite-dimensional setting,  that is,  recovering the unknown conductivity function exactly from infinitely many measurements \cite{ KL, JG}.
 Later, Harrach \cite{HarrachGlobal2019} demonstrated that the conductivity coefficient can be uniquely determined from finitely many measurements by employing the Runge approximation.
 Recently, Fang, Deng, and Liu \cite{Liu2021} showed that the location of a conductive rod can be determined from a single measurement.
Due to the lack of continuous dependence of the solution on the data, the inverse problem is generally unstable, and typically only logarithmic-type stability estimates can be established under standard {\it a priori} assumptions on the conductivity \cite{Beretta1998Stable}.  Furthermore, considerable effort has been devoted to investigating Lipschitz stability under restrictive assumptions on the admissible set of conductivities.
The first Lipschitz stability result for the inverse conductivity problem was established by Alessandrini and Vessella \cite{Alessandrini2005}, who employed all Cauchy data pairs of solutions. This was later improved in \cite{HarrachGlobal2019}, which demonstrated that even a finite number of Cauchy data pairs is sufficient.
Most recently, Hanke \cite{Hanke2024Lipschitz} proved Lipschitz stability for the inverse conductivity problem with only two Cauchy data pairs, and showed that a single pair is sufficient in the case of a polygonal conductivity inclusion.

Numerical reconstruction approaches for the inverse conductivity problem are typically based on minimizing a regularized data-fitting functional.
The minimization problem is formulated as the minimization of a residual functional with a regularization term, subject to a PDE constraint in the form of an elliptic equation\cite{Jin2021, Chen1999}.
A Newton-type method is usually employed to solve the minimization problem, while the finite element method is used to handle the PDE constraint.
Due to the non-convexity of the objective functional, the regularized data-fitting technique usually suffers from local convergence. To overcome this difficulty, Harrach and Minh \cite{Harrach2016Enhancing} proposed a monotonicity-based regularization approach, in which the monotonicity relation serves as a specialized regularizer.  Interested readers could also refer to the globally convergent algorithms\cite{Klibanov2019, Klibanov2025} and dynamical regularization algorithm\cite{Zhang2018}.  Actually,  the aforementioned approaches typically require many iterative steps to converge to a satisfactory result.
To avoid repeated iterations, Huhtala, Bossuyt, and Hunnukainen \cite{ASA} reformulated the minimization problem for the inverse source problem of the Poisson equation as an equivalent variational problem, which was then solved using the finite element method.  Subsequently, this methodology  had  been extended to the inverse source problem for  biharmonic equation \cite{MD}.
However, this finite-element based method cannot be directly applied to the nonlinear inverse conductivity problem, as the PDE constraint cannot be explicitly incorporated into the objective functional. Therefore,  it is of significant interest to develop a non-iterative finite-element based approach for solving nonlinear inverse problems.

In this paper, we propose a novel shape reconstruction approach for the inverse conductivity problem with finitely many measurements, which incorporates the finite element method.  We first reformulate the inverse problem as a minimization problem under a linearized regularized residual functional. To this end, we employ a single-step linearization to establish the connection between the NtD operator and its Fr\'echet derivative in quadratic form \cite{Harrach2010Linearization}, which allows the residual functional to be expressed in terms of their difference. Importantly, this residual functional implicitly enforces the PDE solution, so no additional PDE constraint is required. To address the inherent ill-posedness, we incorporate a Frobenius-norm residual with an additional Tikhonov regularization term.  Next, we transform the minimization problem into a variational formulation and establish their equivalence, where the symmetric bilinear form is defined through the trace of a finite-dimensional matrix. The existence and uniqueness of the variational solution are then proved using the Lax-Milgram theorem.
Finally,  we establish a rigorous {\it a priori} error estimates for the inversion scheme, comprehensively accounting for both the reconstruction and discretization errors.  Interested readers could refer to the error estimates in \cite{Chen2022, Jin2021, Ammari, Griesmaier, B2} and the references therein for further details on inverse conductivity problems.  Since only finitely many measurements are available, we show that the reconstructed object lies in a finite-dimensional space and construct an orthonormal basis for it.  Consequently, the reconstruction error reduces to the error in the coefficients. However,  the Galerkin projection error typically cannot be rigorously estimated due to the intrinsic loss of information associated with finite-dimensional measurements.
To address this challenge,  we demonstrate that,  by selecting input functions from a trigonometric basis, a rigorous error bound can be established by the properties of Zernike polynomials.

The promising features of our proposed finite element method can be summarized in three aspects. First, although we present the regularization term in the $L^2$-norm for convenience, our error-estimation framework is sufficiently general to accommodate a broad class of regularization terms defined over various Sobolev spaces. This flexibility enables the consideration of diverse smoothness properties and structural assumptions regarding the conductivity distribution. Second, conventional finite element-based approaches to nonlinear inverse problems typically require repeated iterations to solve both the forward and inverse problems,  whereas our method directly reconstructs the shape of the unknown conductivity without any iterative procedure.  Finally, a major challenge in regularized inversion lies in the selection of the regularization parameter. Our error estimate provides a clear and quantitative criterion for this choice, indicating that the regularization parameter depends on both the number of measurements and the noise level.  Hence, our method could enhance the reliability and reproducibility of the reconstruction.

The structure of this paper is as follows. In section 2, we start with some fundamental mathematical theory concerning the NtD operator as well as its Fr\'{e}chet derivative. We then demonstrate the equivalence between a  regularized minimizing problem  and  a variational formulation.  Section 3 provides a comprehensive error analysis,  including both the reconstruction error and the discretization error. Based on these estimates, we propose a rule for selecting the regularization parameter that depends only on the number of measurement data and the noise level.
Finally, several numerical examples involving various geometric shapes are presented to verify our theoretical results in Section 4.  %Moreover, the reconstructions obtained by our method achieve comparable quality to those produced by the iterative approach.

\section{Minimization problem and its variational formulation}
In this section, we introduce a finite element method to determine the shape of an unknown conductivity. We first utilize a minimization problem based on a one-step linearization to characterize the inverse conductivity problem. Then we reformulate this minimization problem  as a variational problem, and establish  the equivalence between the two formulations. Finally, the variational problem is solved using the finite element method. Before our discussions, we introduce the necessary notations and Sobolev spaces.

In order to guarantee the uniqueness of  equation \eqref{PDE}, we assume that the electric potential $u$ has zero mean on $\partial \Omega$, i.e., $\int_{\partial \Omega} u \mathrm{d}s=0$. To this end, we define the zero-mean subspaces of $L^2(\partial\Omega)$ and $H^1(\Omega)$ as
\begin{equation*}\label{space1}
\!\!  L^2_{\diamond}(\partial\Omega)=\left\{f\in L^2(\partial\Omega):\int_{\partial\Omega}f\, \mathrm{d}s=0\right\},\quad
H^1_{\diamond}(\Omega)=\left\{u\in H^1(\Omega):\int_{\partial\Omega}u\, \mathrm{d}s=0\right\}.
\end{equation*}
We further assume that $L^{\infty}_+(\Omega)$ denotes the subspace of $L^{\infty}(\Omega)$ consisting of functions with positive essential infima. For each $\sigma\in L^{\infty}_+(\Omega)$,
 equation \eqref{PDE} admits a unique weak solution by the Lax-Milgram theorem, i.e.,
$u_g^\sigma\in H^1_{\diamond}(\Omega)$ solves
\begin{equation}\label{eq:weak-solution}
\nabla\cdot(\sigma\nabla u_g^\sigma)=0, \ \   \text{in}\  \Omega,  \quad \sigma \partial_{\nu} u_g^\sigma=g, \ \  \text{on}\ \partial\Omega.
\end{equation}
Thus, there exists a one-to-one relation between $g$ and $u_g^\sigma|_{\partial\Omega}$, which together form a Cauchy pair. On this basis,  the NtD operator $\Lambda(\sigma)\in \mathcal{L}\left(L^2_{\diamond}(\partial\Omega)\right) $ in weak form is given by
\begin{equation*}
\Lambda(\sigma):  g\in L^2_{\diamond}(\partial\Omega) \longmapsto u_g^\sigma|_{\partial\Omega}\in L^2_{\diamond}(\partial\Omega).
\end{equation*}
It is well known  that $\Lambda(\sigma)$  is a self-adjoint, linear, bounded, and compact operator \cite{A}.  For $g\in  L^2_{\diamond}(\partial\Omega)$, the quadratic form of $\Lambda(\sigma)$ is given by
\begin{equation*}
\langle g,\Lambda(\sigma)g \rangle =\int_{\partial\Omega} g \Lambda(\sigma)  g \, \mathrm{d}s.
%=\int_\Omega \sigma\left|\nabla u_g^\sigma\right|^2\mbox{d}x.
\end{equation*}
It is noted that  the mapping $\sigma\rightarrow\Lambda(\sigma)$ is Fr\'echet differentiable\cite{Lechleiter2008},  and the Fr\'echet derivative $\Lambda'(\sigma)$ at $\sigma$ in the direction $\kappa$ is given by $(\Lambda'(\sigma)\kappa)g=v|_{\partial\Omega}$, where $v\in H^1_{\diamond}(\Omega)$ solves
\begin{equation*}
\nabla\cdot(\sigma\nabla v)=-\nabla\cdot(\kappa\nabla u_g^\sigma), \  \  \text{in}\ \Omega, \quad   \sigma \partial_{\nu} v=-\kappa \partial_{\nu} u_g^\sigma, \ \ \text{on} \ \partial\Omega.
\end{equation*}
and $u_g^\sigma$ is the solution of \eqref{eq:weak-solution}. Using the integration-by-parts formula, the derivative $\Lambda'(\sigma)\kappa$ can also be represented in the quadratic form
\begin{equation}\label{eq:frechet}
\langle g,(\Lambda'(\sigma)\kappa)g \rangle=-\int_\Omega\kappa\left|
\nabla u_g^\sigma\right|^2 \mathrm{d}x.
\end{equation}

Next,  we present the minimization problem for identifying the shape of  unknown conductivity  based on a one-step linearization.
For simplicity, we assume that the background conductivity $\sigma_0\equiv 1$, which is known as {\it a priori} information.
To reconstruct the contrast   $\sigma-1$, we compare the operator $\Lambda(\sigma)$ with the background operator $\Lambda(\sigma_0)$, corresponding to a known background conductivity $\sigma_0 = 1$.
To this end, we apply the  one-step linearization approach proposed in \cite{Harrach2010Linearization}. The key idea is that,  if $\kappa$ is an exact solution to
\begin{equation}\label{equ}
\Lambda'(1)\kappa=\Lambda(\sigma)-\Lambda(1),
\end{equation}
then
\begin{equation*}
\text{supp} \, \kappa=\text{supp}\{\sigma-1\}.
\end{equation*}
For  finitely many measurements, a discretized version of the relation \eqref{equ}  is given by
\begin{equation*}
F'(1)\kappa=F(\sigma)-F(1).
\end{equation*}
Here $F'(1)\kappa \in \mathbb{S}_n\subset \mathbb{R}^{m\times m}$ denotes the Fr\'echet derivative of $F(\sigma)$ at $\sigma=1$ in the direction $\kappa$. Using \eqref{eq:frechet}, this derivative can be written as
 \begin{equation*}
F'(1)\kappa=\left(\int_{\partial \Omega}\,  g_i  (\Lambda'(1)\kappa ) g_j \,  \mathrm{d}s\right)_{i,j=1}^m=-\left(\int_\Omega\kappa\nabla u_{g_i}^0\cdot\nabla u_{g_j}^0 \,  \mathrm{d}x \right)_{i,j=1}^m \in \mathbb{S}_m \subset \mathbb{R}^{m\times m}.
\end{equation*}
For simplicity in the following discussion, we define
\begin{align}
&\label{eq:S}\mathbf{S}(\kappa):=-F'(1)\kappa=-\left(\int_{\partial \Omega}\,  g_i  (\Lambda'(1)\kappa ) g_j \,  \mathrm{d}s\right)_{i,j=1}^m,\\
%= \left(\int_\Omega\kappa\nabla u_{g_i}^0\cdot\nabla u_{g_j}^0 \,  \mbox{d}x\right)_{i,j=1}^m,\\
&\label{eq:V}\mathbf{V}:=F(1)-F(\sigma)=\left(\int_{\partial \Omega}\,  g_i  (\Lambda(1)-\Lambda (\sigma ) ) g_j \,  \mathrm{d}s\right)_{i,j=1}^m.
\end{align}
We note that $\mathbf{S}(\cdot)$ is a linear matrix  and  $\mathbf{V}$ is composed entirely of measurement data. Accordingly, equation \eqref{equ} can be represented by
\begin{equation*}
\mathbf{S}(\kappa)=\mathbf{V}.
\end{equation*}
The exact solution of the above equation cannot be obtained directly. A natural approach is to minimize the residual, which leads to the determination of an approximate solution  $\kappa_r\in L^2(\Omega)$,  via the Tikhonov regularization  problem
\begin{equation}\label{min}
\kappa_r=\mathop{\mbox{arg min}}\limits_{\kappa\in L^2(\Omega)}\|\mathbf{V}-\mathbf{S}(\kappa)\|_F^2 +\alpha b(\kappa,\kappa),
\end{equation}
where $\alpha>0$ is the regularization parameter,  and $b(\cdot,\cdot)$ is a symmetric,  continuous  and coercive bilinear form on $L^2(\Omega)$. Here $\|\cdot\|_F$ denotes the Frobenius norm induced by the matrix trace,  namely,
\begin{equation}\label{eq:trace}
\begin{aligned}
\|\mathbf{V}-\mathbf{S}(\kappa)\|_F^2 &=\mbox{tr}((\mathbf{V}-\mathbf{S}(\kappa))^{\top}(\mathbf{V}-\mathbf{S}(\kappa)))\\
&=\mbox{tr}(\mathbf{V}^{\top}\mathbf{V})-2\mbox{tr}(\mathbf{V}^{\top}\mathbf{S}(\kappa)) +\mbox{tr}(\mathbf{S}(\kappa)^{\top}\mathbf{S}(\kappa)).
\end{aligned}
\end{equation}

%In general, the minimizer $\kappa_r$ of \eqref{min} differs from the exact solution of \eqref{equ}, since information is inevitably lost through discretization and regularization.
%Iterative methods are commonly employed to solve \eqref{min}, but they often require substantial computational effort. To avoid the cost of repeated iteration, we  then

 Now we reformulate the minimization problem \eqref{min} as the following variational problem, which admits a finite element approximation.

\begin{thm}\label{thm2}
The minimization problem $\eqref{min}$ is equivalent to the following variational formulation: find  $\kappa_r\in L^2(\Omega)$, such that
\begin{equation}\label{variation2}
a(\kappa_r,\eta)=l(\eta),\quad\forall \, \eta\in L^2(\Omega),
\end{equation}
where the bilinear form is defined by
\begin{equation}\label{eq:bilinear}
a(\kappa,\eta)=\mbox{tr}(\mathbf{S}(\kappa)^{\top}\mathbf{S}(\eta))+\alpha b(\kappa,\eta),
\end{equation}
and  the linear form is given by
\begin{equation*}
l(\eta)=\mbox{tr}(\mathbf{V}^{\top}\mathbf{S}(\eta)).
\end{equation*}
Moreover,  the variational problem   \eqref{variation2} admits a unique solution.
\end{thm}
The existence and uniqueness of the solution to \eqref{variation2} can be established via the Lax-Milgram theorem. To this end, it is essential to verify the continuity and coercivity of the bilinear form $a(\cdot,\cdot)$ defined in \eqref{eq:bilinear}.
 These properties are summarized in the following lemma.

\begin{lem}\label{lema}
The bilinear form $a(\cdot,\cdot)$ defined in \eqref{eq:bilinear} is continuous and coercive in $L^2(\Omega)$.
\end{lem}
\begin{proof}
Since $b(\cdot,\cdot)$  is continuous and coercive, it remains to verify that  the bilinear mapping
\begin{equation*}
  (\kappa, \eta)\mapsto \mbox{tr}(\mathbf{S}(\kappa)^{\top}\mathbf{S}(\eta)),
\end{equation*}
is  continuous.
For any $m\times m$ matrix $\mathbf{A}=(a_{ij})_{i, j=1}^m$ and $\mathbf{B}=(b_{ij})_{i, j=1}^m$, the Cauchy –Schwarz inequality implies
\begin{equation*}
\begin{aligned}
\left|\mbox{tr}(\mathbf{A}^{\top}\mathbf{B})\right|&=\left|\sum_{i=1}^m \sum_{j=1}^m a_{ij}b_{ij}\right|\\
&\leq\sqrt{\left(\sum_{i=1}^m\sum_{j=1}^m a_{ij}^2\right)\left(\sum_{i=1}^m\sum_{j=1}^m b_{ij}^2\right)}= \sqrt{\mbox{tr}(\mathbf{A}^{\top}\mathbf{A})\mbox{tr}(\mathbf{B}^{\top}\mathbf{B})}.
\end{aligned}
\end{equation*}
Setting $\mathbf{A}=\mathbf{S}(\kappa)$ and $\mathbf{B}=\mathbf{S}(\eta)$, it follows that
\begin{align*}
&\mbox{tr}( \mathbf{S}(\kappa)^{\top}\mathbf{S}(\eta))\\
\leq&\sqrt{\left(\sum_{i=1}^m\sum_{j=1}^m\left(\int_\Omega \kappa \nabla u_{g_i}^0\cdot\nabla u_{g_j}^0\mbox{d}x\right)^2\right) \left(\sum_{i=1}^m\sum_{j=1}^m\left(\int_\Omega \eta \nabla u_{g_i}^0\cdot\nabla u_{g_j}^0\mbox{d}x\right)^2\right)}\\
\lesssim &\sqrt{\int_\Omega \kappa^2\mbox{d}x\int_\Omega \eta^2\mbox{d}x}=\|\kappa\|_{L^2(\Omega) } \|\eta\|_{L^2(\Omega)},
\end{align*}
which verifies that $\mbox{tr}(\mathbf{S}(\cdot)^{\top}\mathbf{S}(\cdot))$ is continuous.
Especially,  the second inequality follows from the  H\"older's inequality. Here, and in what follows, $x\lesssim y$ denotes $x\leq Cy$ with a positive constant $C$.  Therefore, the bilinear form $a(\cdot,\cdot)$ is continuous as it is bounded:
\begin{equation*}
a(\kappa,\eta)=\mbox{tr}(\mathbf{S}(\kappa)^{\top}\mathbf{S}(\eta))+\alpha b(\kappa,\eta)\lesssim \|\kappa\|_{L^2(\Omega)}\|\eta\|_{L^2(\Omega)}.
\end{equation*}
 Moreover, the bilinear form is coercive, since
\begin{equation*}
a(\eta,\eta)\geq\alpha b(\eta,\eta)\geq C_1\|\eta\|_{L^2(\Omega)}^2.
\end{equation*}
with $C_1>0$ independent of $\eta$.
\end{proof}
With Lemma \ref{lema}, we now proceed to  proof Theorem \ref{thm2}.
\begin{proof}[Proof of Theorem \ref{thm2}.]
Let $\kappa_r$ be the minimizer of  problem\eqref{min}. Then for any $\eta\in L^2(\Omega)$, replacing $\kappa_r$ by  $\kappa_r+\eta$ in \eqref{min}, it holds that
\begin{equation*}
\|\mathbf{V}-\mathbf{S}(\kappa_r)\|_F^2+\alpha b(\kappa_r,\kappa_r)\leq \|\mathbf{V}-\mathbf{S}(\kappa_r+\eta)\|_F^2+\alpha b(\kappa_r+\eta,\kappa_r+\eta),\quad\forall\, \eta\in L^2(\Omega).
\end{equation*}
Using the trace identity for the Frobenius norm \eqref{eq:trace}, the last inequality can be rewritten as
\begin{align*}
&\mbox{tr}(\mathbf{V}^{\top}\mathbf{V})-2\mbox{tr}(\mathbf{V}^{\top}\mathbf{S}(\kappa_r)) +\mbox{tr}(\mathbf{S}(\kappa_r)^{\top}\mathbf{S}(\kappa_r))+\alpha b(\kappa_r,\kappa_r)\\
& \quad \leq\mbox{tr}(\mathbf{V}^{\top}\mathbf{V})-2\mbox{tr}(\mathbf{V}^{\top}\mathbf{S}(\kappa_r+\eta)) +\mbox{tr}(\mathbf{S}(\kappa_r+\eta)^{\top}\mathbf{S}(\kappa_r+\eta))+\alpha b(\kappa_r+\eta,\kappa_r+\eta).
\end{align*}
By a straightforward calculation, one can get that
\begin{equation}\label{leq}
2\mbox{tr}((\mathbf{V}-\mathbf{S}(\kappa_r))^{\top}\mathbf{S}(\eta))-2\alpha b(\kappa_r,\eta) \leq\mbox{tr}(\mathbf{S}(\eta)^{\top}\mathbf{S}(\eta))+\alpha b(\eta,\eta).
\end{equation}
Inequality \eqref{leq} is  still not in the form of a variational form. If we regard $\eta$ as the variable in \eqref{leq}, the left-hand side can be interpreted as a continuous linear functional in $\eta$, while the right-hand side defines a bilinear mapping. Specifically, we define
\begin{align*}
&a_1(\eta):=\mbox{tr}((\mathbf{V}-\mathbf{S}(\kappa_r))^{\top}\mathbf{S}(\eta)),\quad a_2(\eta):=\alpha b(\kappa_r,\eta),\\ &b_1(\eta,\eta):= \mbox{tr}(\mathbf{S}(\eta)^{\top}\mathbf{S}(\eta))+\alpha b(\eta,\eta),
\end{align*}
where $a_1$ and $a_2$ are continuous linear functionals and $b_1$ is a bilinear mapping.  According to Riesz representation theorem, there exist $\psi_1,\psi_2\in L^2(\Omega)$ such that
\begin{equation*}
a_1(\eta)=(\psi_1,\eta),\quad a_2(\eta)=(\psi_2,\eta),\quad\forall\,  \eta\in L^2(\Omega),
\end{equation*}
where $(\cdot,\cdot)$ denotes the inner product in $L^2(\Omega)$. Therefore, (\ref{leq}) can be rewritten as
\begin{equation*}\label{psi}
2(\psi_1-\psi_2,\eta)\leq b_1(\eta,\eta),\quad\forall\,  \eta\in L^2(\Omega).
\end{equation*}
Since $\eta$ is arbitrarily, let $\eta=\beta(\psi_1-\psi_2)$,  where $\beta\in \mathbb{R}\backslash\{0\}$. Then  we have
\begin{equation*}
2\beta\|\psi_1-\psi_2\|_{L^2(\Omega)}^2\leq b_1(\psi_1-\psi_2,\psi_1-\psi_2)\leq C_2\beta^2\|\psi_1-\psi_2\|_{L^2(\Omega)}^2,
\end{equation*}
where $C_2>0$ is impendent of $\beta$, $\psi_1$ and $\psi_2$. It follows that
\begin{equation*}
(2\beta-C_2\beta^2)\|\psi_1-\psi_2\|_{L^2(\Omega)}^2\leq 0.
\end{equation*}
By choosing $\beta$ such that $2\beta-C_2\beta^2>0$, we get $\psi_1-\psi_2=0$, i.e., $a_1(\cdot)=a_2(\cdot)$. Consequently,
\begin{equation*}
\mbox{tr}((\mathbf{V}-\mathbf{S}(\kappa_r))^{\top}\mathbf{S}(\eta))=\alpha b(\kappa_r,\eta),\quad\forall \, \eta\in L^2(\Omega),
\end{equation*}
which coincides with  \eqref{variation2}.  In addition, by  Lemma \ref{lema}, equation \eqref{variation2} admits a unique solution $\kappa_r\in L^2(\Omega)$ as guaranteed by the Lax-Milgram theorem.
\end{proof}

%\textcolor{red}{
%\begin{cor}\label{cor}
%If $\kappa_r\in H^s(\Omega)$ is the solution of \eqref{variation2}, then $\kappa_r$ satisfies the estimate:
%\begin{equation*}
%\|\kappa_r\|_{H^s(\Omega)}\leq C_5 \|\mathbf{V}\|_F,
%\end{equation*}
%where $C_5>0$ is independent of $\mathbf{V}$ and $\kappa_r$.
%\end{cor}
%\begin{proof}
%Notice that bilinear form $a(\cdot,\cdot)$ is coercive in variational formulation \eqref{variation2}. Therefore, one can deduce
%\begin{equation*}
%\begin{aligned}
%C_3\|\kappa_r\|_{H^s(\Omega)}^2&\leq a(\kappa_r,\kappa_r) =\mbox{tr}(\mathbf{V}^{\top}\mathbf{S}(\kappa_r))\leq \sqrt{\mbox{tr}(\mathbf{V}^{\top}\mathbf{V}) \mbox{tr}(\mathbf{\mathbf{S}(\kappa_r)}^{\top}\mathbf{\mathbf{S}(\kappa_r)})}\\
%&\leq \sqrt{C_1} \|\mathbf{V}\|_F \|\kappa_r\|_{H^s(\Omega)}.
%\end{aligned}
%\end{equation*}
%The last two inequalities follow from the proof of Lemma \ref{lema}. This completes the proof.
%\end{proof}
%}

To solve the variation equation \eqref{variation2}, we employ a finite element approximation by discretizing the space $L^2(\Omega)$ using a piecewise polynomial finite element space. Let $\mathcal{T}_h$ be a quasi-uniform triangulation of the domain  $\Omega$ parametrized by mesh size $h$,  and $W_h$ be the finite element space
\begin{equation*}
W_h:=\left\{\eta\in L^2(\Omega):\eta|_{T}\in \mathcal{P}_k(T),\quad\forall \, T\in\mathcal{T}_h\right\},
\end{equation*}
where $\mathcal{P}_k$ is the space of polynomials of maximum total order $k$.
Thus,  the discrete form of \eqref{variation2} is to find  $\kappa_r^h\in W_h$ such that
\begin{equation}\label{eq:femsolution}
a(\kappa_r^h,\eta)=l(\eta),\quad\forall\, \eta\in W_h.
\end{equation}

%To solve the variation equation (\ref{variation2}), we turn to the finite element approximation. Let $\mathcal{T}_\tau$ be a quasi-uniform triangulation of the domain.  $\Omega$ parametrized by mesh size $\tau$. Since what we intend to reconstruct is the domain of conductivity rather than the exact value of it, for simple computation, we denote piecewise constant finite element space
%\begin{equation*}
%\mathcal{A}^\tau=\left\{\eta\in L^2(\Omega)\enspace|\enspace\eta|_{T}=\mbox{constant},\quad\forall T\in\mathcal{T}_\tau\right\}
%\end{equation*}
%Then, the discrete problem is
%\begin{equation*}
%a(\kappa_r^\tau,\eta)=l(\eta),\quad\forall\eta\in\mathcal{A}^\tau
%\end{equation*}

\section{Error estimates}
 Let $\kappa_{\text{true}}$  be the exact solution of \eqref{equ},  $\kappa_{r}$ denote the reconstructed solution of  \eqref{variation2}, and $\kappa_r^h$ represent the finite element approximation of  \eqref{eq:femsolution}.  In this section, we  shall establish an error estimate between  the exact solution  $\kappa_{\text{true}}$ and its  finite element approximation $\kappa_r^h$.  It is noted that
 \begin{equation*}
   \|\kappa_{\text{true}}-\kappa_r^h\|_{L^2{(\Omega)}}\leq     \|\kappa_{\text{true}}-\kappa_r\|_{L^2{(\Omega)}}+   \|\kappa_r-\kappa_r^h\|_{L^2{(\Omega)}},
 \end{equation*}
where the first term on the right-hand side represents the reconstruction error, and the second term corresponds to the discretization error. Therefore, in what follows, we discuss these two errors separately.

\subsection{Error of the reconstruction}
 In the first part, we  analyse the reconstruction error between the exact solution  $\kappa_{\text{true}}$ and the reconstructed solution  $\kappa_r$.

We first show that $\kappa_r$  lies in a finite-dimensional space, as it  is determined from finitely many measurements. Without loss of generality,  let the noisy measurements
$\mathbf{V}^{\delta}$  satisfy
\begin{equation}\label{eq:noise}
\mathbf{V}^{\delta}=\mathbf{V}+\mathbf{E}^{\delta},
\end{equation}
where $\mathbf{E}^{\delta}$ is a matrix representing the measurement error with $\|\mathbf{E}^\delta\|_F=\delta\|\mathbf{V}\|_F$. For the reconstructed solution $\kappa_r$,  according to \eqref{equ},  there exists a conductivity  $\tilde{\sigma}$ such that
\begin{equation*}
\Lambda(\tilde{\sigma})-\Lambda(1)=\Lambda'(1)\kappa_r.
\end{equation*}
Using  the last equation and  the definition of $\mathbf{S}(\cdot)$ in \eqref{eq:S},  we can deduce that
\begin{equation*}
\begin{aligned}
\mathbf{S}(\kappa_{\text{true}}-\kappa_r)
&= \left(\int_{\partial \Omega}\,  g_i  (\Lambda'(1)(\kappa_r-\kappa_{\text{true}} ) g_j \,  \mathrm{d}s\right)_{i,j=1}^m\\
&=\left(\int_{\partial \Omega}\,  g_i  (\Lambda(\tilde{\sigma})-\Lambda(\sigma)) g_j \,  \mathrm{d}s\right)_{i,j=1}^m\\
&=\left(\int_{\partial\Omega}g_i \, \xi_j\mbox{d}s\right) _{i,j=1}^m,
\end{aligned}
\end{equation*}
where $\xi_j:=(\Lambda(\tilde{\sigma})-\Lambda(\sigma)) g_j $. According to \eqref{variation2},
by replacing $\mathbf{V}$ with noisy data in $\mathbf{V}^\delta$ \eqref{eq:noise} and applying the previous formula, one can derive that
\begin{equation}\label{eq:b}
\begin{aligned}
b(\kappa_r,\eta)&=\frac{1}{\alpha}\left( \mbox{tr}\left({\mathbf{V}^\delta}^{\top}\mathbf{S}(\eta)\right)- \mbox{tr}\left(\mathbf{S}(\kappa_r)^{\top}\mathbf{S}(\eta)\right)\right)\\
&=\frac{1}{\alpha}\mbox{tr}\left(\left( \mathbf{S}(\kappa_{\text{true}}-\kappa_r)+ \mathbf{E}^\delta\right)^{\top}\mathbf{S}(\eta)\right)\\
&=\frac{1}{\alpha} \mbox{tr}\left(\left(\left(\int_{\partial\Omega}g_j\xi_i\mbox{d}s\right)_{i,j=1}^m +{\mathbf{E}^\delta}^{\top}\right) \left(\int_\Omega\eta\nabla u_{g_i}^0\cdot\nabla u_{g_j}^0\mbox{d}x\right)_{i,j=1}^m\right)\\
&=\frac{1}{\alpha} \sum_{i=1}^m\sum_{j=1}^m\left(\int_{\partial\Omega}g_i\xi_j\mbox{d}s +\mathbf{E}^\delta_{ij}\right) \int_\Omega\eta\nabla u_{g_i}^0\cdot\nabla u_{g_j}^0\mbox{d}x\\
&=\sum_{i=1}^m\sum_{j=1}^m\beta_{ij}\psi_{ij}(\eta),
\end{aligned}
\end{equation}
where
\begin{equation*}
\beta_{ij}:=\frac{1}{\alpha}\left(\int_{\partial\Omega}g_i\xi_j\mbox{d}s +\mathbf{E}^\delta_{ij}\right),\quad
\psi_{ij}(\eta):=\int_\Omega\eta\nabla u_{g_i}^0\cdot\nabla u_{g_j}^0\mbox{d}x.
\end{equation*}
Moreover,  let $\zeta_{ij}$ be the solution of the following  variational equation
\begin{equation}\label{variation4}
b(\zeta_{ij},\eta)=\psi_{ij}(\eta),\quad\forall \, \eta\in L^2(\Omega),
\end{equation}
then the reconstructed solution $\kappa_r$ can be represented as
\begin{equation}\label{eq:kappa1}
\kappa_r=\sum_{i=1}^m \sum_{j=1}^m \beta_{ij}\zeta_{ij},
\end{equation}
which indicates that $\kappa_r$   lies in a finite-dimensional space.

Notice that the set $\{\zeta_{11},\ldots,\zeta_{mm}\}$ is linearly dependent because  $\zeta_{ij}=\zeta_{ji}$. Taking this into account, we choose an orthonormal basis $\{\hat{\zeta}_1,\hat{\zeta}_2,\ldots,\hat{\zeta}_{m'}\}$ for $\mbox{span}\{\zeta_{11},\zeta_{12},\ldots,\zeta_{mm}\}$ with respect to the inner product $b(\cdot,\cdot)$,  such that
\begin{equation}\label{Im}
\left(b(\hat{\zeta}_i,\hat{\zeta_j})\right)_{i,j=1}^{m'}=\mathbf{I}_{m'},
\end{equation}
where $\mathbf{I}_{m'}$ is the $m'$-dimensional identity matrix, and
\begin{equation}\label{zeta}
\mbox{span}\{\hat{\zeta}_1,\hat{\zeta}_2,\ldots,\hat{\zeta}_{m'}\}= \mbox{span}\{\zeta_{11},\zeta_{12},\ldots,\zeta_{mm}\}.
\end{equation}
We introduce the following vector notations:
\begin{equation}\label{eq:zeta}
\zeta:=(\zeta_{11},\ldots,\zeta_{1m},\ldots,\zeta_{m1},\ldots,\zeta_{mm})^{\top},\quad \hat{\zeta}:=(\hat{\zeta}_1,\hat{\zeta}_2,\ldots,\hat{\zeta}_{m'})^{\top}.
\end{equation}
By \eqref{zeta}, the vectors $\zeta$ and $\hat{\zeta}$ can be linearly represented in terms of each other, that is, there exist matrices $\mathbf{T}\in \mathbb{R}^{m^2\times m'}$ and $\hat{\mathbf{T}}\in \mathbb{R}^{m'\times m^2}$, such that
\begin{equation}\label{eq:T}
 \zeta=\mathbf{T}\hat{\zeta}, \quad  \text{and} \quad  \hat{\zeta}=\hat{\mathbf{T}}\zeta.
\end{equation}
Consequently,
\begin{equation*}
(\hat{\mathbf{T}}\mathbf{T}-\mathbf{I}_{m'})\hat{\zeta}=\theta.
\end{equation*}
where $\theta:=(0,\dots,0)^{\top}$ denotes the zero vector. Since $\hat{\zeta}$ forms  an orthonormal  basis, it follows that
\begin{equation*}
\hat{\mathbf{T}}\mathbf{T}=\mathbf{I}_{m'}, \quad  \text{and}  \quad \mbox{rank}(\mathbf{T})=m'.
\end{equation*}
Furthermore,  let $P\kappa_{\text{true}}$ denote the Galerkin  projection of $\kappa_{\text{true}}$ onto $\mbox{span}\{\hat{\zeta}_1,\hat{\zeta}_2,\ldots,\hat{\zeta}_{m'}\}$.
Therefore,  the error between $\kappa_{\text{true}}$ and $\kappa_r$ can be decomposed as
\begin{equation}\label{eq:error1}
\| \kappa_r-\kappa_{\text{true}}\|_{L^2(\Omega)}\leq \|P\kappa_{\text{true}}-\kappa_r\|_{L^2(\Omega)}+ \|\kappa_{\text{true}}-P\kappa_{\text{true}}\|_{L^2(\Omega)}.
\end{equation}

Next,  we estimates the error term $\|P\kappa_{\text{true}}-\kappa_r\|_{L^2(\Omega)}$. Before proceeding,  we present two key lemmas.

\begin{lem}\label{lemgamma}
Let $P\kappa_{\text{true}}$ denote the Galerkin  projection of $\kappa_{\text{true}}$ onto $\mbox{span} \{\hat{\zeta}_1,\hat{\zeta}_2,\ldots,\hat{\zeta}_{m'}\}$, such that
\begin{equation}\label{projection}
b(P\kappa_{\text{true}},\hat{\zeta}_j)= b(\kappa_{\text{true}},\hat{\zeta}_j),\quad j=1,\ldots,{m'}.
\end{equation}
Then the Galerkin  projection $P\kappa_{\text{true}}$ can be represented as
\begin{equation*}
P\kappa_{\text{true}}=\hat{\gamma}^\top\hat{\zeta},
\end{equation*}
where $\hat{\zeta}$ is  the orthonormal  basis defined in \eqref{eq:zeta},  and the coefficient vector $\hat{\gamma}$ is
\begin{equation*}
\hat{\gamma}=\left(\mathbf{T}^{\top}\mathbf{T}\right)^{-1} \mathbf{T}^{\top}\mbox{vec}(\mathbf{V}),
\end{equation*}
with  $\bm T$ defined in \eqref{eq:T}. Here and throughout this paper,   $\mbox{vec}(\cdot)$ denotes the vectorization operator that stacks the entries of a matrix into a column vector.
\end{lem}
\begin{proof}
%In order to determine the coefficients of $P\kappa_{\text{true}}$, we compute both sides of \eqref{projection} separately.
For simplification, we set
\begin{align*}
b(\zeta,\hat{\zeta}_j)&:=(b(\zeta_{11},\hat{\zeta}_j),\ldots, b(\zeta_{1m},\hat{\zeta}_j),\ldots,b(\zeta_{m1},\hat{\zeta}_j), \ldots,b(\zeta_{mm},\hat{\zeta}_j))^\top, \\
b(\hat{\zeta},\hat{\zeta}_j)&:=(b(\hat{\zeta}_1,\hat{\zeta}_j), b(\hat{\zeta}_2,\hat{\zeta}_j),\ldots,b(\hat{\zeta}_{m'},\hat{\zeta}_j))^\top.
\end{align*}
By substituting $P\kappa_{\text{true}}=\hat{\gamma}^\top\hat{\zeta}$ into  the left-hand side of \eqref{projection},  and using
\begin{equation}\label{eq:I}
\mathbf{I}_{m'}=(b(\hat{\zeta},\hat{\zeta}_1), b(\hat{\zeta},\hat{\zeta}_2),\ldots,b(\hat{\zeta},\hat{\zeta}_{m'})),
\end{equation}
as defined in \eqref{Im},  we  can derive that
\begin{equation}\label{eq:left1}
\left(b(P\kappa_{\text{true}},\hat{\zeta}_1), \ldots,  b(P\kappa_{\text{true}},\hat{\zeta}_{m'})\right)^{\top}
=\left( b(\hat{\zeta},\hat{\zeta}_1), \ldots,  b(\hat{\zeta},\hat{\zeta}_{m'})  \right)^{\top} \hat{\gamma}
=\hat{\gamma}.
\end{equation}

Moreover, from \eqref{eq:b} and \eqref{variation4}, one can find that
\begin{equation}\label{eq:S-eta}
  \mathbf{S}(\eta)=\left(\psi_{ij}(\eta)\right)_{i,j=1}^m =\left(b(\zeta_{ij},\eta)\right)_{i,j=1}^m.
\end{equation}
Multiplying the right-hand side of \eqref{projection} by the matrix $\mathbf{T}$ and using the previous equation, we obtain
\begin{equation}\label{eq:right1}
\begin{aligned}
\mathbf{T}\left (b(\kappa_{\text{true}},\hat{\zeta}_1), \ldots, b(\kappa_{\text{true}},\hat{\zeta}_{m'})\right)^{\top}
&= \left ( b(\kappa_{\text{true}},\zeta_{11}), \ldots,  b(\kappa_{\text{true}},\zeta_{mm}) \right)^{\top}\\
&=\mbox{vec}(\mathbf{S}(\kappa_{\text{true}}))=\mbox{vec}(\mathbf{V}).
\end{aligned}
\end{equation}
Hence, combining \eqref{eq:left1} and \eqref{eq:right1}, one has
\begin{equation}\label{gamma}
\mathbf{T}\hat{\gamma}=\mbox{vec}(\mathbf{V}).
\end{equation}
Since $\mbox{rank}(\mathbf{T})=m'$,
there exists a unique solution $\hat{\gamma}$ of the system \eqref{gamma}, which is given by
\begin{equation*}
\hat{\gamma}=\left(\mathbf{T}^\top\mathbf{T}\right)^{-1} \mathbf{T}^\top\mbox{vec}(\mathbf{V}).
\end{equation*}

This completes the proof.
\end{proof}

\begin{lem}\label{lembeta}
Let $\mathbf{T}$ be defined as in \eqref{eq:T}, and let $\hat{\zeta}=\{\hat{\zeta}_1,\hat{\zeta}_2,\ldots,\hat{\zeta}_{m'}\}$ be the orthonormal basis. Then  $\kappa_r$  defined as in \eqref{variation2}  can be expressed as
\begin{equation}\label{Kappa_r}
\kappa_r=\hat{\beta}^\top \hat{\zeta},
\end{equation}
where  $\hat{\beta}:=(\hat{\beta}_1,\hat{\beta}_2,\ldots,\hat{\beta}_{m'})^{\top}$ can be represented as
\begin{equation*}
\hat{\beta}=\left(\alpha\mathbf{I}_{m'}+\mathbf{T}^{\top}\mathbf{T}\right)^{-1} \mathbf{T}^{\top} \mbox{vec}(\mathbf{V}^{\delta}).
\end{equation*}
\end{lem}

\begin{proof}
Using the the noisy data $\mathbf{V}^\delta$ as the measurement,  substituting  \eqref{Kappa_r} into \eqref{variation2},
one has
\begin{equation}\label{variation3}
\sum_{i=1}^{m'}\hat{\beta}_i\mbox{tr}\left(\mathbf{S}(\hat{\zeta}_i)^{\top}\mathbf{S} (\hat{\zeta}_j)\right) +\alpha\sum_{i=1}^{m'}\hat{\beta}_i b(\hat{\zeta}_i,\hat{\zeta}_j)= \mbox{tr}\left((\mathbf{V}^{\delta})^{\top}\mathbf{S}(\zeta_j)\right), \quad j=1,\ldots,m'.
\end{equation}
Using the definition $\zeta=\mathbf{T}\hat{\zeta}$,  we can get that $b(\zeta,\hat{\zeta}_j)=\mathbf{T}b(\hat{\zeta},\hat{\zeta}_j)$.
Combining this with \eqref{eq:I} and  \eqref{eq:S-eta}, the first term on the left-hand side of \eqref{variation3} can be written as
\begin{equation*}
\begin{aligned}
\sum_{i=1}^{m'}\hat{\beta}_i\mbox{tr}\left(\mathbf{S}(\hat{\zeta}_i)^{\top} \mathbf{S}(\hat{\zeta}_j)\right)
&=\sum_{i=1}^{m'}\hat{\beta}_i\sum_{k=1}^m \sum_{l=1}^m b(\zeta_{kl},\hat{\zeta}_i)b(\zeta_{kl},\hat{\zeta}_j)\\
&=\sum_{i=1}^{m'} b(\zeta,\hat{\zeta}_j)^{\top} b(\zeta,\hat{\zeta}_i)\hat{\beta}_i\\
&=\sum_{i=1}^{m'} b(\hat{\zeta},\hat{\zeta}_j)^{\top}\mathbf{T}^{\top}\mathbf{T} b(\hat{\zeta},\hat{\zeta}_i)\hat{\beta}_i\\
&=b(\hat{\zeta},\hat{\zeta}_j)^{\top}\mathbf{T}^{\top}\mathbf{T} \sum_{i=1}^{m'}b(\hat{\zeta},\hat{\zeta}_i)\hat{\beta}_i\\
&=b(\hat{\zeta},\hat{\zeta}_j)^{\top}\mathbf{T}^{\top}\mathbf{T}\hat{\beta}.
\end{aligned}
\end{equation*}
Similarly, the right-hand side of \eqref{variation3} yields that
\begin{align*}
\mbox{tr}\left((\mathbf{V}^{\delta})^{\top}\mathbf{S}(\zeta_j)\right)&= \sum_{k=1}^m\sum_{l=1}^m\mathbf{V}_{kl}^{\delta}  b(\zeta_{kl},\hat{\zeta}_j)\\
&=b(\zeta,\hat{\zeta}_j)^{\top}\mbox{vec}(\mathbf{V}^{\delta})\\
&=b(\hat{\zeta},\hat{\zeta}_j)^{\top}\mathbf{T}^{\top}\mbox{vec}(\mathbf{V}^{\delta}).
\end{align*}
Therefore, \eqref{variation3} can be rewritten as
\begin{equation*}
b(\hat{\zeta},\hat{\zeta}_j)^{\top}\mathbf{T}^{\top}\mathbf{T}\hat{\beta}+ \alpha b(\hat{\zeta},\hat{\zeta}_j)^{\top}\hat{\beta}=b(\hat{\zeta},\hat{\zeta}_j)^{\top}\mathbf{T}^{\top} \mbox{vec}(\mathbf{V}^{\delta}),\quad j=1,\ldots,m',
\end{equation*}
that is,
\begin{equation*}
\mathbf{T}^{\top}\mathbf{T}\hat{\beta}+\alpha\hat{\beta}=\mathbf{T}^{\top} \mbox{vec}(\mathbf{V}^{\delta}).
\end{equation*}
Recall that $\mbox{rank}(\mathbf{T})=m'$, which implies that $\alpha\mathbf{I}_{m'}+\mathbf{T}^{\top}\mathbf{T}$ is invertible with $\alpha>0$.  Hence, the coefficients $\hat{\beta}$ can be represented by
\begin{equation*}
\hat{\beta}=\left(\alpha\mathbf{I}_{m'}+\mathbf{T}^{\top}\mathbf{T}\right)^{-1} \mathbf{T}^{\top} \mbox{vec}(\mathbf{V}^{\delta}).
\end{equation*}

This completes the proof.
\end{proof}

\begin{rem}
From Lemma \ref{lemgamma}, it follows that the Galerkin projection $P\kappa_{\text{true}}$ depends only on the finite set of measurements $\mathbf{V}$.  In contrast, Lemma \ref{lembeta} concerns $\kappa_r$, which is obtained from the solution of \eqref{variation2} and is influenced both by the regularization parameter $\alpha$ and by the noisy measurement data $\mathbf{V}^{\delta}$.

\end{rem}

With $\hat{\beta}$ and $\hat{\gamma}$ determined, we can now estimate $\|P\kappa_{\text{true}}-\kappa_r\|_{L^2(\Omega)}$.

\begin{thm}\label{Pk}
Let the conditions in Lemma \ref{lemgamma} and \ref{lembeta} hold,  then we have
\begin{equation*}
\|P\kappa_{\text{true}}-\kappa_r\|_{L^2(\Omega)}\lesssim  \|\mathbf{V}\|_F\sqrt{\alpha^2+\frac{\delta^2}{\alpha^2} +\left(1+\frac{1}{\alpha^2}\right)\delta}.
\end{equation*}
\end{thm}
\begin{proof}
Through coercivity of the bilinear form $b(\cdot,\cdot)$,  one can get
\begin{align*}
C_1\|P\kappa_{\text{true}}-\kappa_r\|_{L^2(\Omega)}^2
\leq&b(P\kappa_{\text{true}}-\kappa_r,P\kappa_{\text{true}}-\kappa_r)\\
=&b\left(\sum_{i=1}^{m'}(\hat{\gamma}_i-\hat{\beta}_i)\hat{\zeta}_i, \sum_{i=1}^{m'}(\hat{\gamma}_i-\hat{\beta}_i)\hat{\zeta}_i\right)\\
=&\left(\hat{\gamma}-\hat{\beta}\right)^\top\left(\hat{\gamma}-\hat{\beta}\right).
\end{align*}
With Lemma  \ref{lemgamma}   and \ref{lembeta} ,  one can deduce that
\begin{equation*}
  \hat{\gamma}-\hat{\beta}=\left(\mathbf{T}^{\top}\mathbf{T}\right)^{-1} \mathbf{T}^{\top}\mbox{vec}(\mathbf{V})- \left(\alpha\mathbf{I}_{m'}+\mathbf{T}^{\top}\mathbf{T}\right)^{-1} \mathbf{T}^{\top} \left(\mbox{vec}(\mathbf{V})+\mbox{vec}(\mathbf{E}^\delta)\right).
\end{equation*}
Let $(\lambda, x)$ be the eigenmode of $\mathbf{T}^{\top}\mathbf{T}$ with $\mathbf{T}^{\top}\mathbf{T}x=\lambda x$, then we readily obtain
\begin{equation*}
\begin{aligned}
&\left(\mathbf{T}^{\top}\mathbf{T}\right)^{-2}x=\frac{1}{\lambda^2}x,\quad \left(\alpha\mathbf{I}_{m'}+\mathbf{T}^{\top}\mathbf{T}\right)^{-2}x =\frac{1}{(\lambda+\alpha)^2}x,\\
&\left(\mathbf{T}^{\top}\mathbf{T}\right)^{-1} \left(\alpha\mathbf{I}_{m'}+\mathbf{T}^{\top}\mathbf{T}\right)^{-1}x =\frac{1}{\lambda(\lambda+\alpha)}x.
\end{aligned}
\end{equation*}
By a straightforward calculation, it yields
\begin{equation}\label{error}
\begin{aligned}
& \left(\hat{\gamma}-\hat{\beta}\right)^{\top}\left(\hat{\gamma}-\hat{\beta}\right)\\
&\quad\leq\max_{\lambda\in\lambda(\mathbf{T}^{\top}\mathbf{T})}\left\{\frac{1}{\lambda^2}+ \frac{1}{(\lambda+\alpha)^2}-\frac{2}{\lambda(\lambda+\alpha)}\right\} \|\mathbf{T}^{\top}\mbox{vec}(\mathbf{V})\|_{l^2}^2\\
&\quad \quad+\max_{\lambda\in\lambda(\mathbf{T}^{\top}\mathbf{T})} \left\{\frac{1}{(\lambda+\alpha)^2}\right\} \|\mathbf{T}^{\top}\mbox{vec}(\mathbf{E}^\delta)\|_{l^2}^2\\
& \quad \quad +2\max_{\lambda\in\lambda(\mathbf{T}^{\top}\mathbf{T})} \left\{\frac{1}{(\lambda+\alpha)^2}+\frac{1}{\lambda(\lambda+\alpha)}\right\} \|\mathbf{T}^{\top}\mbox{vec}(\mathbf{V})\|_{l^2} \|\mathbf{T}^{\top}\mbox{vec}(\mathbf{E}^\delta)\|_{l^2}\\
&\quad\leq\frac{\alpha^2\lambda_{1}}{\lambda_{m'}^2(\lambda_{m'}+\alpha)^2}\|\mathbf{V}\|_F^2+ \frac{\lambda_{1}}{(\lambda_{m'}+\alpha)^2}\|\mathbf{E}^\delta\|_F^2 +\frac{(4\lambda_{m'}+2\alpha)\lambda_{1}}{\lambda_{m'}(\lambda_{m'}+\alpha)^2} \|\mathbf{V}\|_F\|\mathbf{E}^\delta\|_F\\
&\quad \lesssim \left(\alpha^2+\frac{\delta^2}{\alpha^2} +\left(1+\frac{1}{\alpha^2}\right)\delta\right)\|\mathbf{V}\|_F^2,
\end{aligned}
\end{equation}
where $\lambda_1$ is the biggest eigenvalue of matrix $\mathbf{T}^{\top}\mathbf{T}$ and $\lambda_{m'}$ is the smallest one. Hence,  we obtain the error estimate
\begin{equation*}
\|P\kappa_{\text{true}}-\kappa_r\|_{L^2(\Omega)}\lesssim\|\mathbf{V}\|_F\sqrt{\alpha^2+\frac{\delta^2}{\alpha^2} +\left(1+\frac{1}{\alpha^2}\right)\delta}.
\end{equation*}

This completes the proof.
\end{proof}

In general, the  projection error $\|P\kappa_{\text{true}}-\kappa_{\text{true}}\|_{L^2(\Omega)} $ is difficult to estimate  without additional measurements, since information is inevitably lost when finitely many measurements $\mathbf{V}$ are collected.  Nevertheless, an error estimate can be derived under certain assumptions.
Let the computational domain be the unit disk in $\mathbb{R}^2$, i.e., $\Omega=\{(x,y):x^2+y^2\leq1\}$. We specifically employ a set of orthonormal trigonometric functions, considering the current densities $g_j$ in the following orthonormal set of $L_\diamond^2(\partial\Omega)$:
\begin{equation}\label{eq:tri}
\left\{\frac{1}{\sqrt{\pi}}\sin(j\phi), \frac{1}{\sqrt{\pi}}\cos(j\phi): j=1,2,\cdots,n\right\},\quad m=2n.
\end{equation}
Under this setting, we can obtain the following projection error.

\begin{thm}\label{thm:projection}
Let the boundary current densities $\{g_j\}_{j=1}^m$ be the orthonormal trigonometric set  defined in \eqref{eq:tri},
and let $P\kappa_{\text{true}}$ denote the Galerkin  projection of $\kappa_{\text{true}}$ onto $\mbox{span} \{\hat{\zeta}_1,\hat{\zeta}_2,\ldots,\hat{\zeta}_{m'}\}$, defined by
\begin{equation*}
b(P\kappa_{\text{true}},\hat{\zeta}_j)= b(\kappa_{\text{true}},\hat{\zeta}_j),\quad j=1,\ldots,{m'},
\end{equation*}
where $b(\cdot,\cdot)$ is the inner product in $L^2(\Omega)$.
If $\kappa_{\text{true}}\in \mathcal{B}\cap H^s(\Omega)$, then it holds that
\begin{equation*}
\|P\kappa_{\text{true}}-\kappa_{\text{true}}\|_{L^2(\Omega)} \lesssim \left(\frac{m-2}{2}\right)^{-s} \|\kappa_{\text{true}}\|_{H^s(\Omega)}, \quad0\leq s\leq \frac{m}{2}.
\end{equation*}

\end{thm}

\begin{proof}
Since $\{g_j\}_{j=1}^m$ is the orthonormal trigonometric set  defined in \eqref{eq:tri}, setting $\sigma=1$,
the solution to  \eqref{eq:weak-solution} with the Neumann boundary \eqref{eq:tri} are given  by
\begin{equation*}
u_{g_j}^0=\left\{
\begin{aligned}
&\frac{1}{j\sqrt{\pi}}\sin(j\phi)r^j,\quad\mbox{if}\enspace g_j=\frac{1}{\sqrt{\pi}}\sin(j\phi),\\
&\frac{1}{j\sqrt{\pi}}\cos(j\phi)r^j,\quad\mbox{if}\enspace g_j=\frac{1}{\sqrt{\pi}}\cos(j\phi),
\end{aligned}\right.
\end{equation*}
with  gradients
\begin{equation*}
\nabla u_{g_j}^0=\left\{
\begin{aligned}
&\frac{r^{j-2}}{\sqrt{\pi}}
\begin{pmatrix}
\sin{j\phi}&-\cos{j\phi}\\
\cos{j\phi}&\sin{j\phi}
\end{pmatrix}
\begin{pmatrix}
x\\y
\end{pmatrix}
,\quad\mbox{if}\enspace g_j=\frac{1}{\sqrt{\pi}}\sin(j\phi),\\
&\frac{r^{j-2}}{\sqrt{\pi}}
\begin{pmatrix}
\cos{j\phi}&\sin{j\phi}\\
-\sin{j\phi}&\cos{j\phi}
\end{pmatrix}
\begin{pmatrix}
x\\y
\end{pmatrix}
,\quad\mbox{if}\enspace g_j=\frac{1}{\sqrt{\pi}}\cos(j\phi).
\end{aligned}\right.
\end{equation*}
Since $b(\cdot,\cdot)$ is the inner product in $L^2(\Omega)$, the functions $\zeta_{ij}$ defined in \eqref{variation4}  can be explicitly solved as
\begin{equation*}
\zeta_{ij}=\frac{r^{i+j-2}}{\pi} \sin{(i-j)\phi} \ \ \mbox{or} \ \ \zeta_{ij}= \frac{r^{i+j-2}}{\pi} \cos{(i-j)\phi},\quad 1\leq i,j\leq n, \ m=2n.
\end{equation*}
Noting that the set $\{\zeta_{ij} \mid 1\leq i,j\leq n, i+j\leq n+1\}$ can be expressed as linear combinations of Zernike polynomials with radial order less than $n-1$ \cite{Lakshminarayanana}. Since the Zernike polynomials constitute a complete orthogonal basis of
$L^2$ functions on the unit disk, together with
\begin{equation*}
\mbox{span}\{\hat{\zeta}_1,\hat{\zeta}_2,\ldots,\hat{\zeta}_{m'}\}= \mbox{span}\{\zeta_{11},\zeta_{12},\ldots,\zeta_{mm}\},
\end{equation*}
  it follows that  $\mbox{span}\{\hat{\zeta}_1,\hat{\zeta}_2,\ldots,\hat{\zeta}_{m'}\}$ contains all polynomials of degree less than  $n-1$.
Thus,  according to \cite[Remark 3.7]{jieshen2011}  the Galerkin projection error of $\kappa_{\text{true}}$ can be estimated as
\begin{equation*}
\|P\kappa_{\text{true}}-\kappa_{\text{true}}\|_{L^2(\Omega)} \lesssim  \left(\frac{m-2}{2}\right)^{-s} \|\kappa_{\text{true}}\|_{H^s(\Omega)}, \quad0\leq s\leq \frac{m}{2}.
\end{equation*}
\end{proof}

\begin{rem}
Theorem \ref{thm:projection} implies that
\begin{equation*}
\|P\kappa_{\text{true}}-\kappa_{\text{true}}\|_{L^2(\Omega)}\rightarrow 0 \quad \text{as} \quad m\rightarrow \infty.
\end{equation*}
This indicates that the error introduced by the Galerkin projection decreases as the number of measurements increases.
\end{rem}

Under the conditions of Theorem \ref{thm:projection}, the Neumann boundary condition
$g_j$  can be expressed as
\begin{equation*}
g_j=\left\{
\begin{aligned}
&\frac{1}{\sqrt{\pi}}\sin\left(\frac{j+1}{2}\phi\right), \quad \, \text{if $j$ is odd,} \\
&\frac{1}{\sqrt{\pi}}\cos\left(\frac{j}{2}\phi\right),  \quad \quad \ \  \text{if $j$ even,}
\end{aligned}
\right.\quad j=1,\ldots,m.
\end{equation*}
and the corresponding $\zeta_{ij}$ can be  represented as
\begin{equation*}
\zeta_{ij} = \left\{
\begin{aligned}
&\frac{r^{\frac{i+j-2}{2}}}{\pi}\cos\left(\frac{i-j}{2}\phi\right),&\text{if $i,j$ are odd,}\\
&\frac{r^{\frac{i+j-3}{2}}}{\pi}\sin\left(\frac{i-j+1}{2}\phi\right),&\text{if $i$ is odd and $j$ is even,}\\
&\frac{r^{\frac{i+j-3}{2}}}{\pi}\sin\left(\frac{j-i+1}{2}\phi\right),&\text{if $i$ is even and $j$ is odd,}\\
&\frac{r^{\frac{i+j-4}{2}}}{\pi}\cos\left(\frac{i-j}{2}\phi\right),&\text{if $i,j$ are even.}
\end{aligned}
\right.
\end{equation*}
According to the Zernike polynomial representation, we define the orthonormal basis  $\hat{\zeta}$ as
\begin{equation*}
\hat{\zeta}=\left(\frac{1}{\sqrt{2\pi}},\sqrt{\frac{3}{\pi}}r\sin\phi, \ldots,\sqrt{\frac{2m-3}{\pi}}r^{m-2}\right)^{\top}.
\end{equation*}
Using the definition of  $\zeta$ in \eqref{eq:zeta} and the relation $\mathbf{T}\hat \zeta=\zeta$ in \eqref{eq:T}, a straightforward calculation yields
\begin{equation}\label{eq:TT}
\mathbf{T}^{\top}\mathbf{T}=
\begin{pmatrix}
\frac{4}{\pi}\\
& \frac{4}{3\pi}\\
& & \frac{4}{3\pi}\\
& & & \frac{4}{5\pi}\\
& & & & \ddots\\
& & & & & \frac{4}{(2m-5)\pi}\\
& & & & & &\frac{4}{(2m-3)\pi}
\end{pmatrix}_{m'\times m'},
\end{equation}
where $m'={m^2}/{4}$ and $\mathbf{T}$ is given by

\begin{equation*}
\mathbf{T}=
\begin{pmatrix}
0& & & \\
0& & & \\
\sqrt{\frac{2}{\pi}}& & & \\
\sqrt{\frac{2}{\pi}}& & & \\
 &-\frac{1}{\sqrt{3\pi}}& & \\
 &-\frac{1}{\sqrt{3\pi}}& & \\
 &\frac{1}{\sqrt{3\pi}}& & \\
 &\frac{1}{\sqrt{3\pi}}& & \\
&  &\frac{1}{\sqrt{3\pi}}& & \\
& &\frac{1}{\sqrt{3\pi}}& & \\
& &\frac{1}{\sqrt{3\pi}}& & \\
& &\frac{1}{\sqrt{3\pi}}& & \\
& & & 0\\
& & & 0\\
& & & \sqrt{\frac{2}{5\pi}}\\
& & & \sqrt{\frac{2}{5\pi}}\\
& & & &\ddots &\\
& & & & &\frac{1}{\sqrt{(2m-5)\pi}}\\
& & & & &\frac{1}{\sqrt{(2m-5)\pi}}\\
& & & & &\frac{1}{\sqrt{(2m-5)\pi}}\\
& & & & &\frac{1}{\sqrt{(2m-5)\pi}}\\
& & & & & &0\\
& & & & & &0\\
& & & & & &\sqrt{\frac{2}{(2m-3)\pi}}\\
& & & & & &\sqrt{\frac{2}{(2m-3)\pi}}
\end{pmatrix}_{m^2\times m'}.
\end{equation*}
From \eqref{eq:TT},  it is clear to see that  the $m'$ eigenvalues of $\mathbf{T}^{\top}\mathbf{T}$ are arranged in descending order.
Consequently, the largest and smallest eigenvalues are
\begin{equation}\label{eq:lambda}
\lambda_1=\frac{4}{\pi},  \quad\lambda_{m'}=\frac{4}{(2m-3)\pi}.
\end{equation}

% Denote by $\mathcal{B}$ the function space to which $\kappa_{\text{true}}$ belongs. Then the operator $P:\mathcal{B} \rightarrow \mbox{span}\{\hat{\zeta}_1,\hat{\zeta}_2,\ldots,\hat{\zeta}_{m'}\}$ is a $L^2$-orthogonal projection satisfying
%\begin{equation*}
%|P\kappa_{\text{true}}-\kappa_{\text{true}}|_{H^s(\Omega)}\leq C_7 \sqrt{\frac{(m-k-1)!}{(m-s-1)!}}(m+k-2)^{(s-k)/2}|\kappa_{\text{true}}|_{H^k(\Omega)},\quad 0\leq s\leq k\leq m-1,
%\end{equation*}
%provided $\kappa_{\text{true}}\in \mathcal{B}\cap H^k(\Omega)$. Consequently, the Galerkin projection error of $\kappa_{\text{true}}$ can be bounded as
%\begin{equation*}
%\|P\kappa_{\text{true}}-\kappa_{\text{true}}\|_{L^2(\Omega)}\leq C_8 m^{-k} |\kappa_{\text{true}}|_{H^k(\Omega)},
%\end{equation*}
%where $C_7$ and $C_8$ are positive constant which are independent of $m$ and $\kappa_{\text{true}}$. \textcolor{red}{From the above estimate, if $k\geq1$, it follows that
%\begin{equation*}
%\|P\kappa_{\text{true}}-\kappa_{\text{true}}\|_{L^2(\Omega)}\rightarrow 0 \quad \text{as} \quad m\rightarrow \infty,
%\end{equation*}
%which implies that as more measurements are acquired, the error introduced by the Galerkin projection diminishes.}

\subsection{ Error of the discretization}  In the second  part, we analyse the discretization error
between the reconstructed solution $\kappa_r$ and its numerical approximation $\kappa_r^h$ obtained via the finite element method.

\begin{thm}\label{thm1}
Let $\kappa_r\in H^{k+1}(\Omega),k\geq 0$ be the solution of \eqref{variation2} such that  $\|\kappa_r\|_{H^{k+1}(\Omega)} \lesssim \|\mathbf{V}\|_F$, and let $\kappa_r^h\in W_h$ be the solution of \eqref{eq:femsolution}. Then it  holds that
\begin{equation*}
\|\kappa_r-\kappa_r^h\|_{L^2(\Omega)} \lesssim h^{k+1} \|\mathbf{V}\|_F.
\end{equation*}
\end{thm}
\begin{proof}
We define an interpolation operator $I^h:H^{k+1}(\Omega)\rightarrow W_h$,  one can find that
$(I^h \kappa_r-\kappa_r^h)\in W_h$. This follows that
\begin{equation*}
\begin{aligned}
C_1\|\kappa_r-\kappa_r^h\|_{L^2(\Omega)}^2&\leq a(\kappa_r-\kappa_r^h,\kappa_r-\kappa_r^h)\\
&=a(\kappa_r-\kappa_r^h,\kappa_r-I^h \kappa_r) +a(\kappa_r-\kappa_r^\tau,I^h \kappa_r-\kappa_r^h)\\
&=a(\kappa_r-\kappa_r^h,\kappa_r-I^h \kappa_r)\\
&\lesssim\|\kappa_r-\kappa_r^h\|_{L^2(\Omega)}\|\kappa_r-I^h \kappa_r\|_{L^2(\Omega)},
\end{aligned}
\end{equation*}
together with  $\|\kappa_r-I^h \kappa_r\|_{L^2(\Omega)} \lesssim  h^{k+1}\|\kappa_r\|_{H^{k+1}(\Omega)}$
in \cite{SL}, one obtains
\begin{equation*}
\|\kappa_r-\kappa_r^h\|_{L^2(\Omega)}\lesssim  h^{k+1} \|\kappa_r\|_{H^{k+1}(\Omega)} \lesssim h^{k+1}\|\mathbf{V}\|_F.
\end{equation*}
\end{proof}

The main result of this section is stated in the following theorem.
\begin{thm}\label{thm:main}
Assume that the bilinear form $b(\cdot,\cdot)$ is chosen as the inner product in $L^2(\Omega)$ and a special set of orthonormal trigonometric functions \eqref{eq:S} is employed as the boundary input in \eqref{variation4}.
If $\kappa_{\text{true}}\in H^s(\Omega)$ with $0\leq s\leq\frac{m}{2}$ is the exact solution of \eqref{equ}, $\kappa_r \in H^{k+1}(\Omega)$ is the solution of \eqref{variation2} and $\kappa_r^h\in W_h$ is the solution of \eqref{eq:femsolution}, then the following overall error estimate holds:

\begin{equation*}
\begin{aligned}
\|\kappa_{\text{true}}&-\kappa_r^h\|_{L^2(\Omega)}
\\
 \lesssim & \left( \left(\frac{m-2}{2}\right)^{-s} \|\kappa_{\text{true}}\|_{H^s(\Omega)}+ \left(h^{k+1}+\sqrt{\alpha^2+\frac{\delta^2}{\alpha^2} +\left(1+\frac{1}{\alpha^2}\right)\delta}\right)\|\mathbf{V}\|_F\right).
\end{aligned}
\end{equation*}

\end{thm}
\begin{proof}
According to Theorem \ref{Pk},  \ref{thm:projection} and  \ref{thm1},  we have
\begin{equation*}
\begin{aligned}
\|\kappa_{\text{true}}&-\kappa_r^h\|_{L^2(\Omega)}\leq \|\kappa_{\text{true}}-P\kappa_{\text{true}}\|_{L^2(\Omega)}+ \|P\kappa_{\text{true}}-\kappa_r\|_{L^2(\Omega)}+ \|\kappa_r-\kappa_r^h\|_{L^2(\Omega)}\\
\lesssim&\left( \left(\frac{m-2}{2}\right)^{-s} \|\kappa_{\text{true}}\|_{H^s(\Omega)}+ \left(h^{k+1}+\sqrt{\alpha^2+\frac{\delta^2}{\alpha^2} +\left(1+\frac{1}{\alpha^2}\right)\delta}\right)\|\mathbf{V}\|_F\right).
\end{aligned}
\end{equation*}
\end{proof}
\begin{rem}\label{rem:alpha}
To derive a more precise error estimate and establish a criterion for selecting the regularization parameter $\alpha$, we define $h_\delta(\alpha)$ as the penultimate term in equation \eqref{error}, that is,
\begin{align*}
h_\delta(\alpha)
&=\frac{\alpha^2\lambda_{1}}{\lambda_{m'}^2(\lambda_{m'}+\alpha)^2}\|\mathbf{V}\|_F^2+ \frac{\lambda_{1}}{(\lambda_{m'}+\alpha)^2} \delta ^2 \|\mathbf{V}\|_F^2 +\frac{(4\lambda_{m'}+2\alpha)\lambda_{1}}{\lambda_{m'}(\lambda_{m'}+\alpha)^2}  \delta\|\mathbf{V}\|_F^2\\
&=\lambda_1\left(\frac{(1+\delta)^2}{(\alpha+\lambda_{m'})^2} +\frac{2(\delta-1)}{\lambda_{m'}(\alpha+\lambda_{m'})} +\frac{1}{\lambda_{m'}^2}\right)\|\mathbf{V}\|_F^2.
\end{align*}
By differentiating $h_\delta(\alpha)$, one obtains that $h_\delta(\alpha)$ attains its minimum  when the noise level satisfies $\delta<1$. Using formula  \eqref{eq:lambda},  the value of $\alpha$ that minimizes the error is given by
\begin{equation}\label{eq:alpha}
\alpha=\frac{\lambda_{m'}(3+\delta)\delta} {1-\delta}=\frac{4(3+\delta)\delta} {(2m-3)\pi(1-\delta)}.
\end{equation}
Thus, we obtain an optimal criterion for selecting the regularization parameter $\alpha$.
Hence,  using  the above $\alpha$, the total error estimate in Theorem \ref{thm:main} can be represented as
\begin{equation*}
\|\kappa_{\text{true}}-\kappa_r^h\|_{L^2(\Omega)}
\lesssim  \left( \left(\frac{m-2}{2}\right)^{-s} \|\kappa_{\text{true}}\|_{H^s(\Omega)}+ \left(h^{k+1}+ \frac{(2m-3)\sqrt{\pi\delta}}{(1+\delta)}\right)\|\mathbf{V}\|_F\right).
\end{equation*}
\end{rem}

\section{Numerical experiments}
In this section, several numerical examples are presented to illustrate the effectiveness of our proposed method. Here, we employ the orthonormal trigonometric set $\{g_j\}_{j=1}^m$ defined in \eqref{eq:tri} as the boundary current densities. The corresponding background electric potential and its gradient are given in Theorem \ref{thm:projection}.

\subsection{Numerical method}
Since our goal is to reconstruct the support of the conductivity rather than its precise values, we adopt a piecewise constant finite element space for computational simplicity. A natural choice of basis is the set of characteristic functions.
We begin by partitioning the domain $\Omega$ into $M$ disjoint open partitions $\{P_j,\text{diam}P_j\leq h\}_{j=1}^M$,  such that
\begin{equation*}
\bar{\Omega}=\bigcup_{j=1}^M \bar{P}_j,\quad P_i\bigcap P_j=\emptyset, \ i\neq j.
\end{equation*}
The finite element space $W_h$ is then defined as
\begin{equation*}
W_h=\text{span}\{\chi_{P_1},\chi_{P_2},\cdots,\chi_{P_M}\},
\end{equation*}
where $\chi_{P_j}$ is the characteristic function of subset $P_j$, given by
\begin{equation*}
\chi_{P_j}(x)=\left\{
\begin{aligned}
&1,\quad x\in P_j,\\
&0,\quad x\in\Omega\backslash P_j,
\end{aligned}\right.
\quad j=1,\ldots,M.
\end{equation*}
Let $\kappa_r^h\in W_h$ be the finite element approximation of $\kappa_r$ defined in  \eqref{eq:femsolution}, which is given by
\begin{equation*}
\kappa_r^h=\sum_{j=1}^M\mu_j\chi_{P_j}.
\end{equation*}
Substituting this expansion into \eqref{eq:femsolution}, it yields the equivalent formula
\begin{equation*}
a(\kappa_r^h,\chi_{P_j})=l(\chi_{P_j}),\quad j=1,\ldots,M.
\end{equation*}
The coefficient vector $\mu:=(\mu_1,\ldots,\mu_M)^T$ is therefore determined by the linear system:
\begin{equation}\label{FEM}
\mathbf{B}\mu=L,
\end{equation}
where the stiff matrix $\mathbf{B}$ and the load vector $L$ are defined by
\begin{equation*}
\mathbf{B}=\left(a(\chi_{P_i},\chi_{P_j})\right)_{i,j=1}^M,\quad L=(l(\chi_{P_1}),\ldots,l(\chi_{P_M}))^{\top}.
\end{equation*}
In our numerical example, the bilinear form $b(\cdot,\cdot)$ is chosen as the inner product in $L^2(\Omega)$:
\begin{equation*}
b(f,g)=\int_\Omega fg\, \mbox{d}x,\quad \forall f, g  \in L^2(\Omega).
\end{equation*}
Using the characteristic basis functions $\chi_{P_i}$, the stiffness matrix $\mathbf{B}$ and load vector $L$ can be computed as follows:
\begin{align*}
&a(\chi_{P_i},\chi_{P_j}) =\mbox{tr}\left(\mathbf{S}(\chi_{P_i})^{\top}\mathbf{S}(\chi_{P_j})\right) +\alpha\int_\Omega\chi_{P_i}\chi_{P_j}\mbox{d}x\\
&\qquad  \qquad \ \ \, =\sum_{k=1}^m\sum_{l=1}^m\int_{P_i}\nabla u_{g_k}^0\cdot\nabla u_{g_l}^0\mbox{d}x\int_{P_j}\nabla u_{g_k}^0\cdot\nabla u_{g_l}^0\mbox{d}x+\alpha\delta_{ij}|P_i|,\\
&l(\chi_{P_j})=\mbox{tr}\left(\mathbf{V}^{\top}\mathbf{S}(\chi_{P_j})\right) =\sum_{k=1}^{m}\sum_{l=1}^{m}\mathbf{V}_{kl}\int_{P_j}\nabla u_{g_k}^0\cdot\nabla u_{g_l}^0\mbox{d}x,
\end{align*}
where $\delta_{ij}$ is the Kronecker delta and $|P_i|$ denotes the area of partition $P_i$. For simplification, we denote the integral
\begin{equation*}
A_{kl}^j=\int_{P_j}\nabla u_{g_k}^0\cdot\nabla u_{g_l}^0\mbox{d}x,
\end{equation*}
and define the matrix $\mathbf{A}$ and $\mathbf{P}$ as
\begin{equation*}
\mathbf{A}=
\begin{pmatrix}
A_{11}^1&A_{11}^2&\cdots&A_{11}^M\\
\vdots&\vdots& &\vdots\\
A_{1m}^1&A_{1m}^2&\cdots&A_{1m}^M\\
\vdots&\vdots& &\vdots\\
A_{m1}^1&A_{m1}^2&\cdots&A_{m1}^M\\
\vdots&\vdots& &\vdots\\
A_{mm}^1&A_{mm}^2&\cdots&A_{mm}^M\\
\end{pmatrix},\quad
\mathbf{P}=
\begin{pmatrix}
|P_1|& & & \\
 &|P_2|& & \\
 & &\ddots& \\
 & & & |P_M|
\end{pmatrix}.
\end{equation*}
The stiffness matrix $\mathbf{B}$ and load vector $L$ can then be expressed in matrix form as
\begin{equation*}
\mathbf{B}=\mathbf{A}^{\top}\mathbf{A}+\alpha\mathbf{P},\quad L=\mathbf{A}^{\top}\mbox{vec}(\mathbf{V}).
\end{equation*}
Thus, the linear system \eqref{FEM} becomes
\begin{equation*}
\left(\mathbf{A}^{\top}\mathbf{A}+\alpha\mathbf{P}\right)\mu =\mathbf{A}^{\top}\mbox{vec}(\mathbf{V}),
\end{equation*}
and the coefficient vector $\mu$ is given by
\begin{equation*}
\mu=\left(\mathbf{A}^{\top}\mathbf{A}+\alpha\mathbf{P}\right)^{-1} \mathbf{A}^{\top}\mbox{vec}(\mathbf{V}).
\end{equation*}

\begin{figure}
\centering

\begin{minipage}[b]{\textwidth}
\centering
\subfigure[exact shape]{
\includegraphics[width=0.3\textwidth]{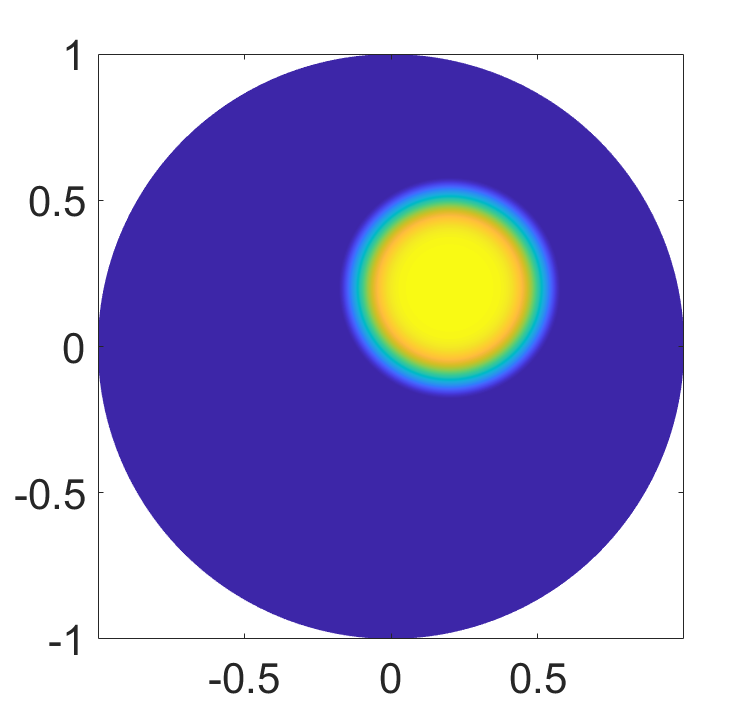}}
\subfigure[$\alpha=10^{-1}$]{
\includegraphics[width=0.3\textwidth]{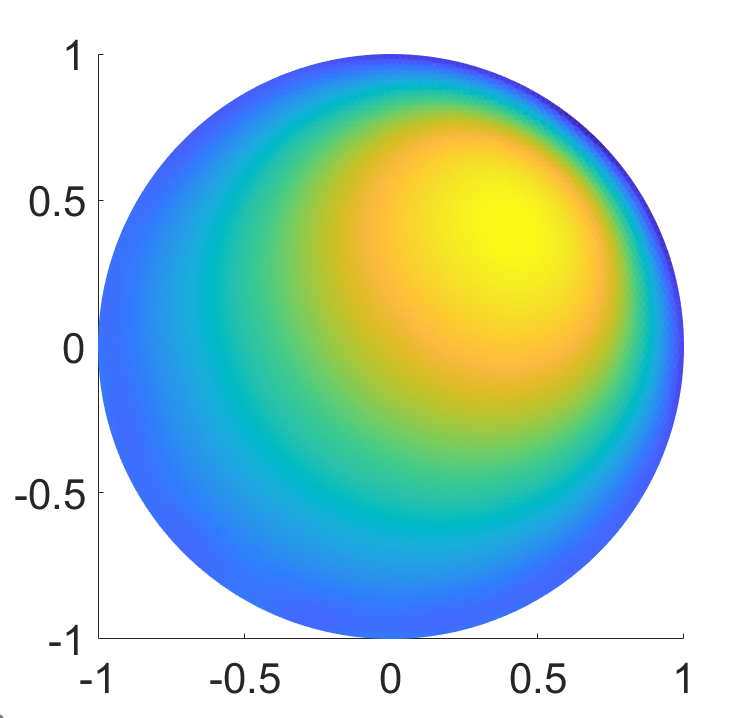}}\\
\subfigure[$\alpha=10^{-5}$]{
\includegraphics[width=0.3\textwidth]{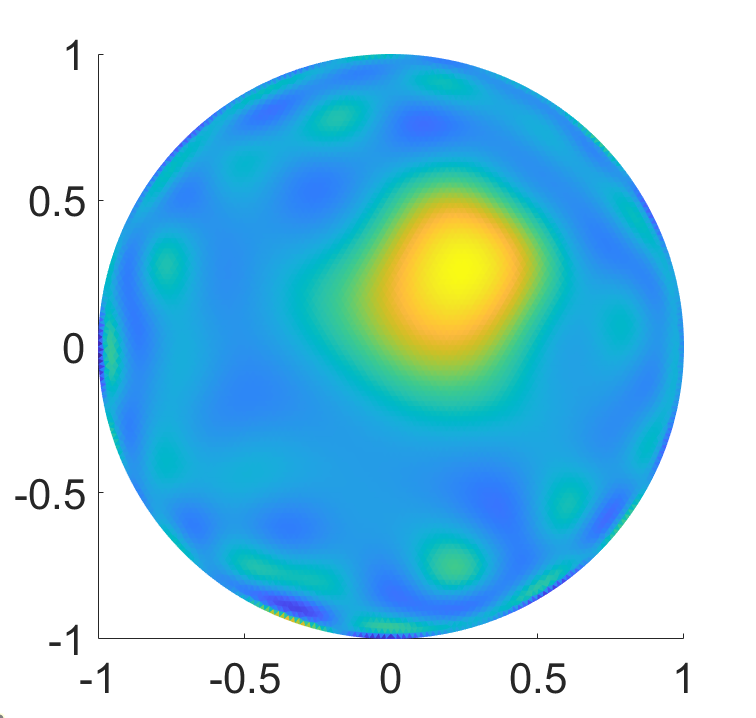}}
\subfigure[$\alpha=3.3506\times 10^{-3}$]{
\includegraphics[width=0.3\textwidth]{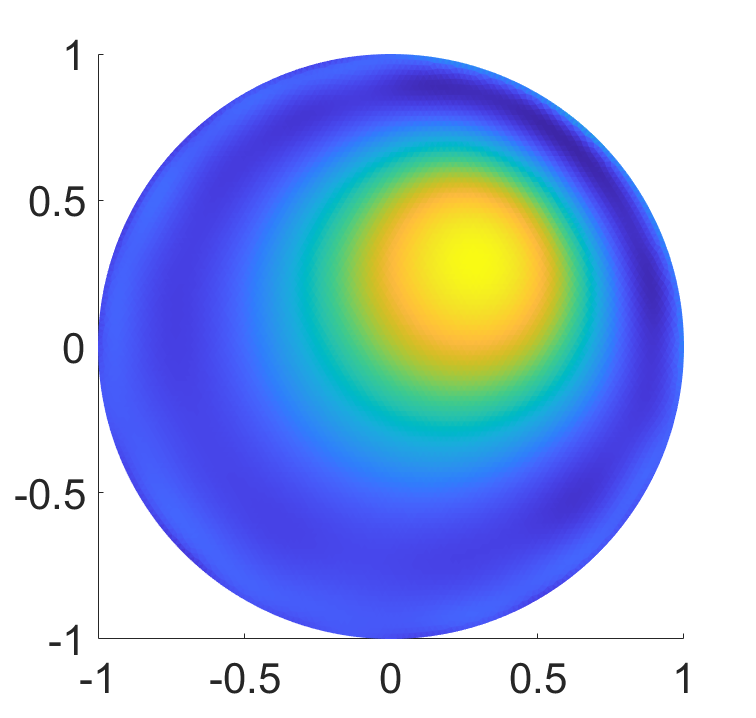}}
\caption{\label{fig-circle}Reconstruction of a smooth circular object under different regularization parameter $\alpha$, with noise level $\delta=5\%$.}
\end{minipage}

\vspace{0.5cm}

\begin{minipage}[b]{\textwidth}
\centering
\subfigure[exact shape]{
\includegraphics[width=0.3\textwidth]{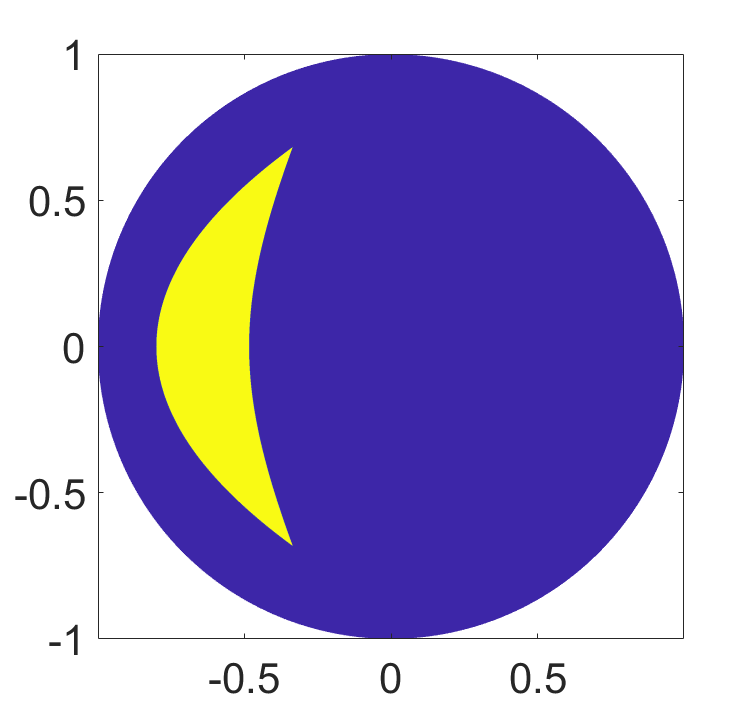}}
\subfigure[$\alpha=10^{-2}$]{
\includegraphics[width=0.3\textwidth]{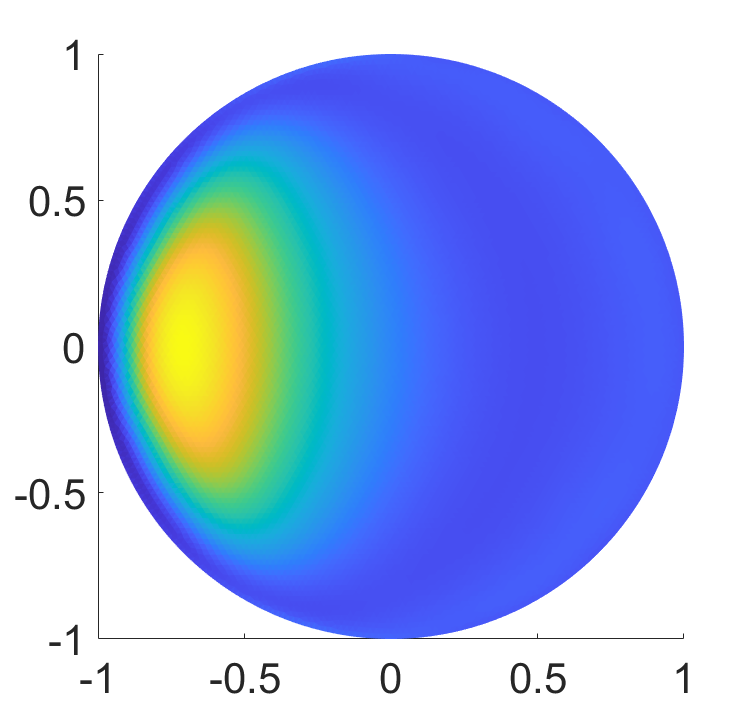}}\\
\subfigure[$\alpha=10^{-6}$]{
\includegraphics[width=0.3\textwidth]{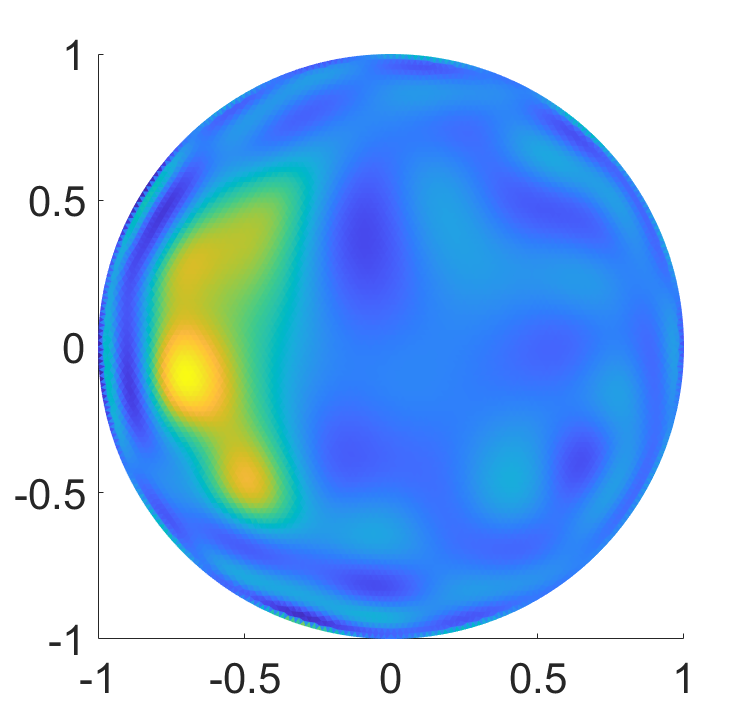}}
\subfigure[$\alpha=6.3462\times 10^{-4}$]{
\includegraphics[width=0.3\textwidth]{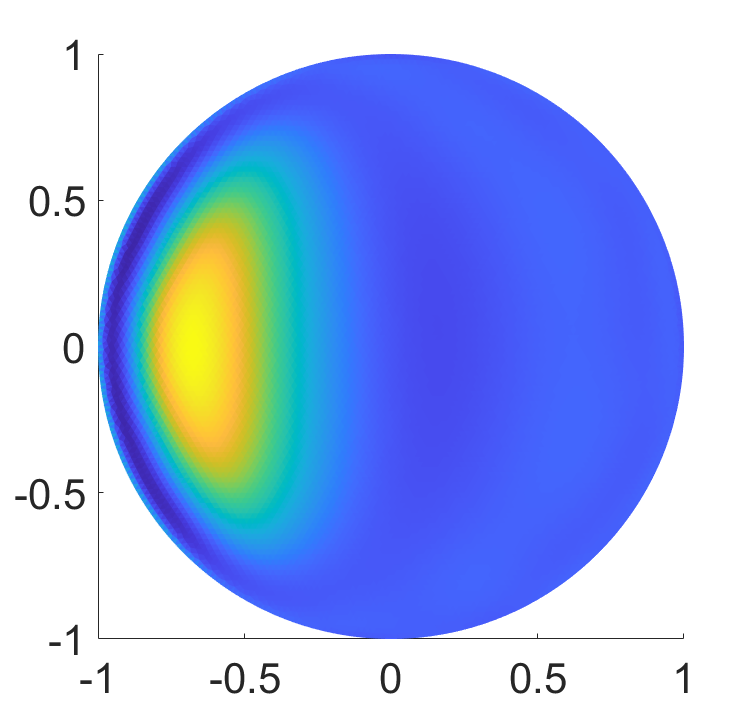}}
\caption{\label{fig-arch} Reconstruction of an arch-shaped object  under different regularization parameter $\alpha$, with noise level $\delta=1\%$.}
\end{minipage}
\end{figure}

\subsection{Numerical examples}

In the following numerical examples, we set the number of boundary currents to $m=32$, and employ a triangular mesh with mesh size $h=0.02$.  Unless otherwise stated, the regularization parameter $\alpha$ is chosen according to \eqref{eq:alpha}, which only depends on the number of boundary currents  $m$ and the noise level $\delta$.

\begin{example}
In the first example, we investigate reconstructions obtained using different values of the regularization parameter
$\alpha$. Figure \ref{fig-circle} presents the reconstructed continuous conductivity for various choices of
$\alpha$, while Figure \ref{fig-arch} shows the corresponding piecewise constant reconstructions.
These results indicate that selecting a regularization parameter that is too small or too large leads to inaccurate reconstructions.
To address this issue, we use  the regularization parameter selecting criterion $\alpha$ given in \eqref{eq:alpha}. From \ref{fig-circle}(d) and \ref{fig-arch}(d), one can find that our method demonstrate a good performance under this special regularization parameter.
\end{example}

\begin{figure}
\centering
\subfigure[exact shape]{
\includegraphics[width=0.3\textwidth]{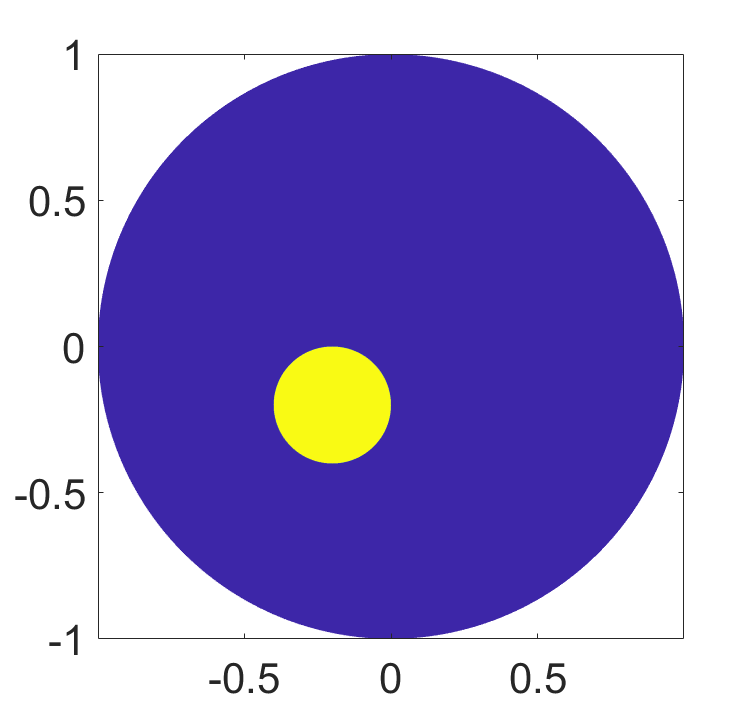}}
\subfigure[iterative method]{
\includegraphics[width=0.3\textwidth]{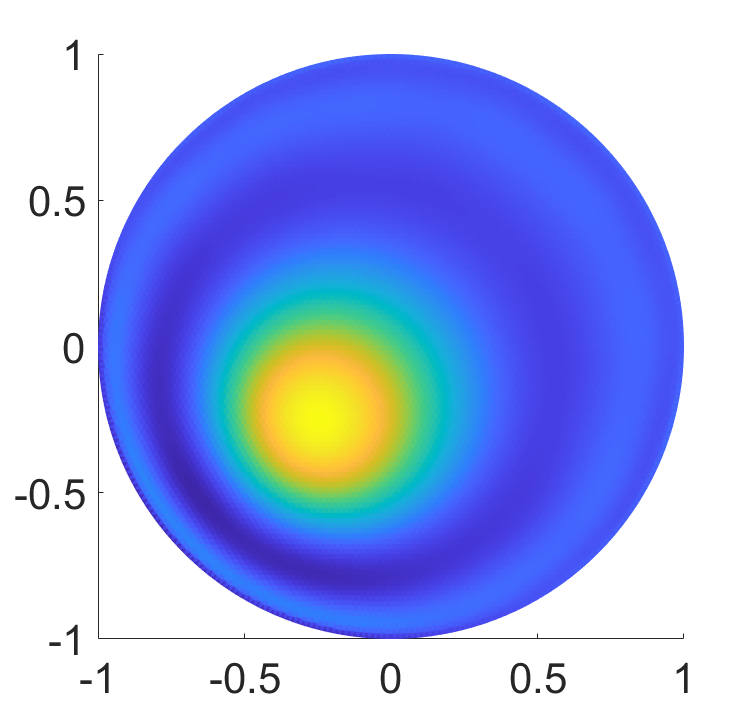}}
\subfigure[finite element method]{
\includegraphics[width=0.3\textwidth]{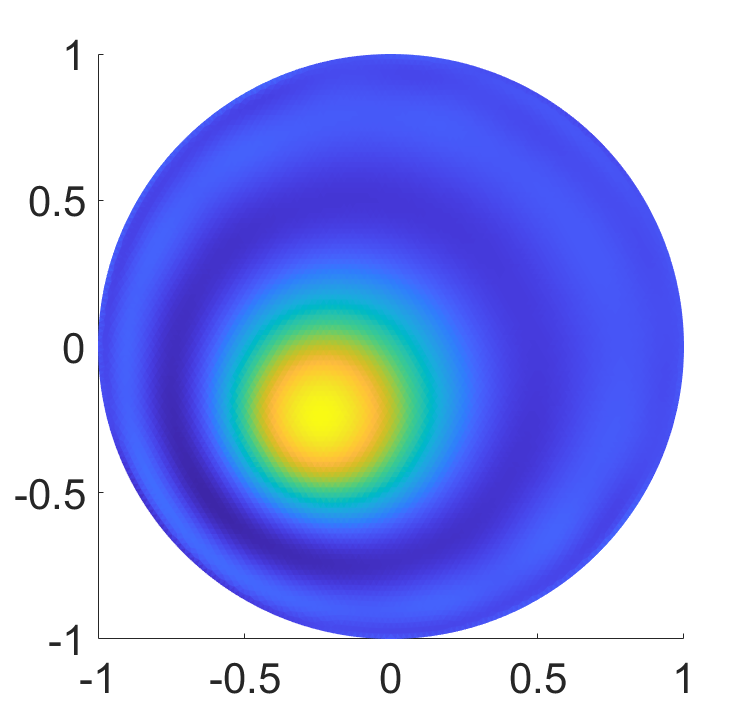}}\\
\subfigure[exact shape]{
\includegraphics[width=0.3\textwidth]{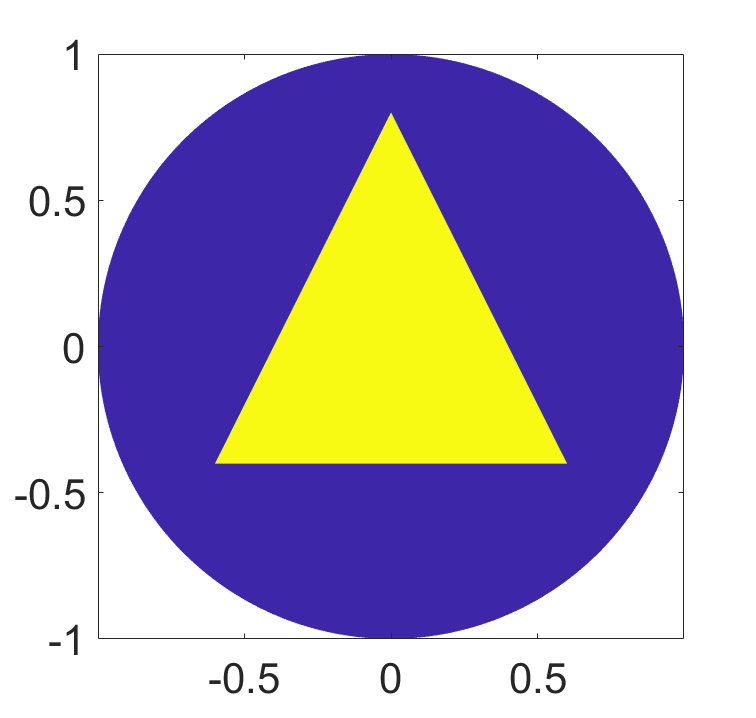}}
\subfigure[iterative method]{
\includegraphics[width=0.3\textwidth]{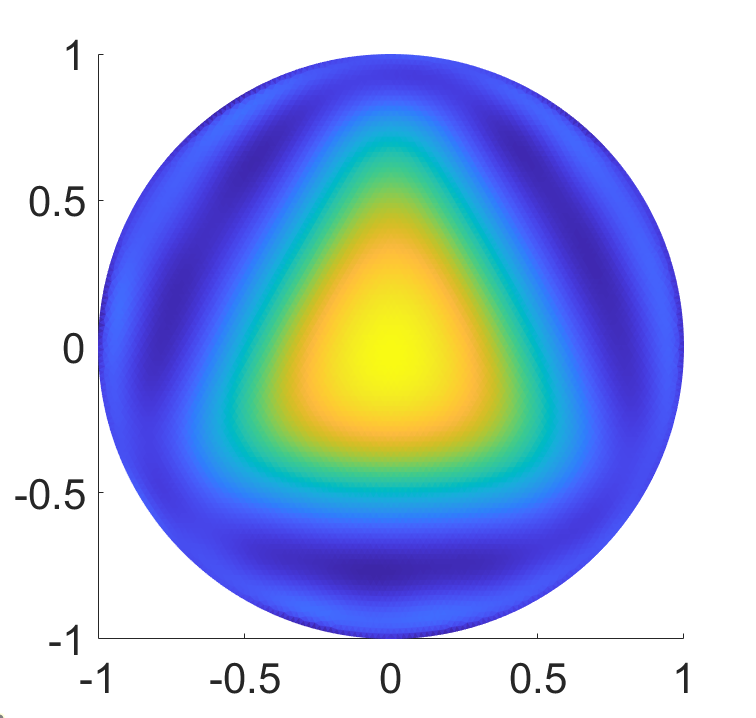}}
\subfigure[finite element method]{
\includegraphics[width=0.3\textwidth]{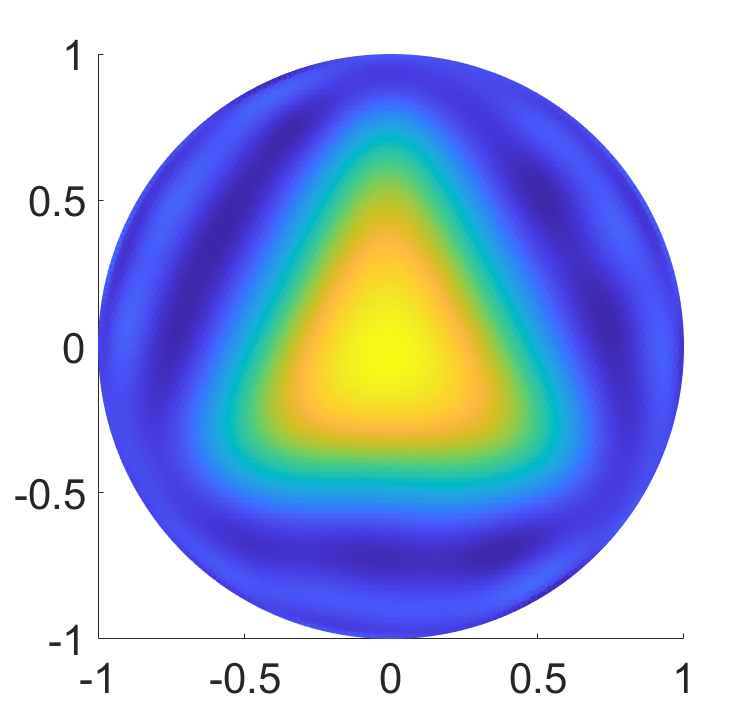}}
\caption{\label{circle} Reconstruction results obtained using the iterative method and the finite element method.}
\end{figure}

\begin{example}
In the second example, we compare our proposed finite element method   with  the iterative method in \eqref{min},
\begin{equation*}
\kappa_r=\mathop{\mbox{arg min}}\limits_{\kappa\in L^2(\Omega)}\|\mathbf{V}-\mathbf{S}(\kappa)\|_F^2 +\alpha \|\kappa\|_{L^2(\Omega)}.
\end{equation*}
To ensure consistency, the noise level is set to $\delta=1\%$, thereby the corresponding regularization parameter computed as $\alpha=6.3462\times 10^{-4}$.  From Figure \ref{circle},  we  can observe that both the  iterative method and the finite element method yield satisfactory reconstruction results. However, the finite element method does not require any iterations or an initial guess, allowing it to directly determine the shape of the conductivity. Moreover, by comparing the center and right columns of Figure \ref{circle}, one can see that the finite element method achieves better resolution in characterizing the boundaries of the object.

\end{example}

\begin{figure}
\centering
\subfigure[exact shape]{
\includegraphics[width=0.3\textwidth]{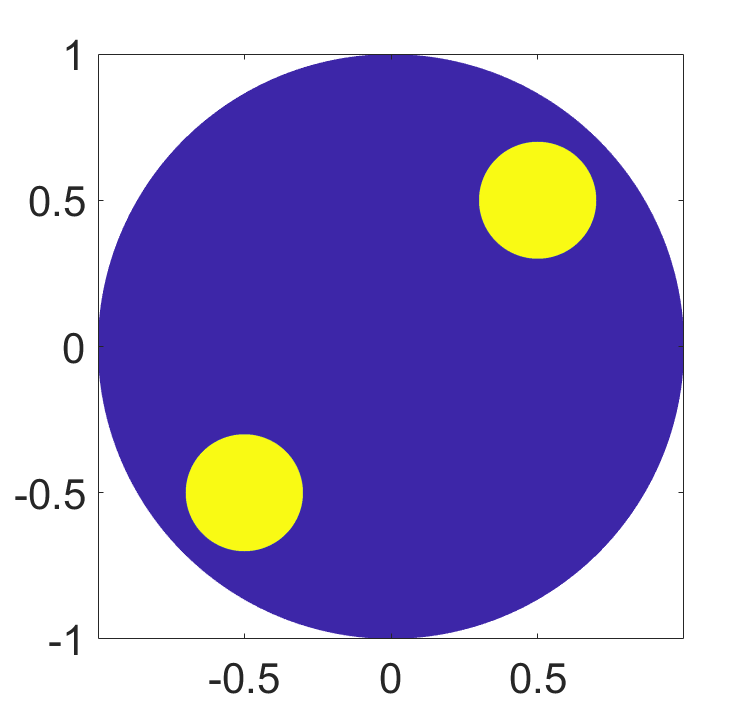}}
\subfigure[reconstruction]{
\includegraphics[width=0.3\textwidth]{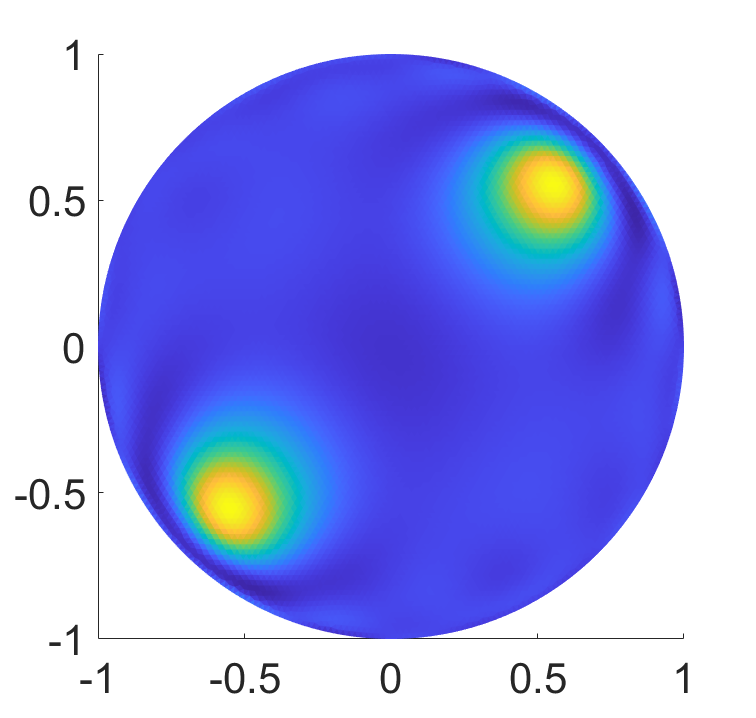}}\\
\subfigure[exact shape]{
\includegraphics[width=0.3\textwidth]{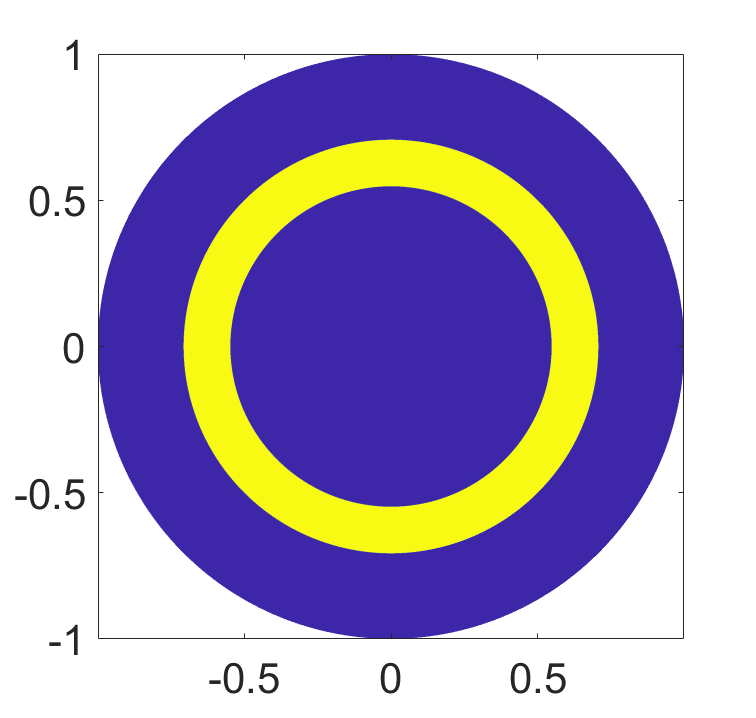}}
\subfigure[reconstruction]{
\includegraphics[width=0.3\textwidth]{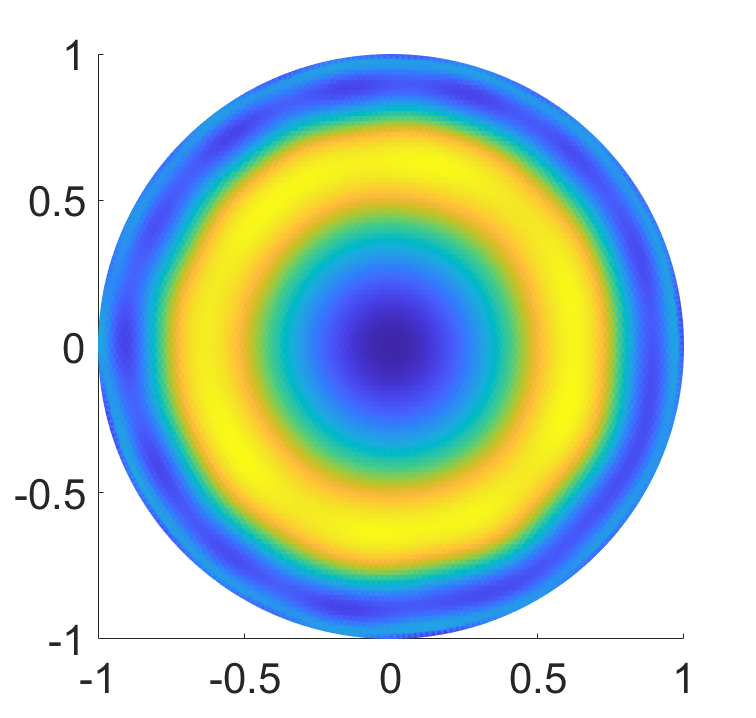}}
\caption{\label{shape} Reconstruction of a loop-shaped object and two circle-shape object.}
\end{figure}

\begin{example}
In the previous examples, we demonstrated that the proposed method recovers conductivity distributions containing single inclusions of various geometric shapes.  In the final example, we consider the finite element method for reconstructing more complex objects. Here the noise level is chosen as $\delta=1\%$, and the corresponding regularization parameter is $\alpha=6.3462\times 10^{-4}$.  Figure \ref{shape} demonstrates that the method also remains effective for shape reconstructions involving disconnected or multiply connected domains.
\end{example}

\end{document}